\documentclass{amsart}
\usepackage{amsmath,amssymb,amsthm}
\usepackage{enumitem}
\usepackage{fixltx2e}
\usepackage{mathtools}
\usepackage{tikz}
\usepackage{tikz-cd}
\usepackage[textsize=small]{todonotes}
\usepackage{xparse}
\usepackage{xspace}
\usepackage{mathtools}
\usepackage[all]{xy}
\usepackage{color}
\usepackage{verbatim}
\usepackage{hyperref}
\usepackage{mathrsfs}
\usepackage{microtype}

\numberwithin{equation}{section}
\mathtoolsset{showonlyrefs=true}

\let\cal\mathcal

\def\Cscr{{\cal C}}
\def\Dscr{{\cal D}}

\def\Fscr{{\cal F}}

\def\Oscr{{\cal O}}

\def\Uscr{{\cal U}}

\let\blb\mathbb

\def \ZZ{{\blb Z}}

\def\id{\text{id}}
\def\Id{\operatorname{id}}

\def\Lotimes{\overset{L}{\otimes}}

\def\Mod{\operatorname{Mod}}
\def\mod{\operatorname{mod}}
\def\Gr{\operatorname{Gr}}

\def\rad{\operatorname {rad}}

\def\GL{\operatorname {GL}}

\def\Ext{\operatorname {Ext}}

\def\End{\operatorname {End}}

\def\im{\operatorname {im}}

\def\coker{\operatorname {coker}}
\def\ker{\operatorname {ker}}
\def\Ker{\operatorname {ker}}
\def\Tor{\operatorname {Tor}}
\def\End{\operatorname {End}}

\def\id{{\operatorname {id}}}

\def\add{\operatorname {add}}

\def\r{\rightarrow}

\def\d{\downarrow}
\def\u{\uparrow}

\def\GL{\operatorname {GL}}

\DeclareMathOperator{\res}{res}
\DeclareMathOperator{\ann}{ann}
\DeclareMathOperator{\eval}{ev}
\DeclareMathOperator{\coeval}{coev}
\DeclareMathOperator{\Ob}{Ob}

\let\invlim\projlim

\theoremstyle{definition}

\newtheorem{lemma}{Lemma}[section]
\newtheorem{proposition}[lemma]{Proposition}
\newtheorem{theorem}[lemma]{Theorem}
\newtheorem{corollary}[lemma]{Corollary}

\newtheorem{lemmas}{Lemma}[subsection]
\newtheorem{propositions}[lemmas]{Proposition}
\newtheorem{theorems}[lemmas]{Theorem}
\newtheorem{corollarys}[lemmas]{Corollary}

\newtheorem{example}[lemma]{Example}

\newtheorem{examples}[lemmas]{Example}
\newtheorem{definitions}[lemmas]{Definition}

\newtheorem{remark}[lemma]{Remark}
\newtheorem{remarks}[lemmas]{Remark}

\DeclareMathOperator\Hom{Hom}

\DeclareMathOperator\comod{Comod}
\DeclareMathOperator\Vect{Vect}
\DeclareMathOperator\VVect{\textbf{Vect}}

\DeclareMathOperator\coend{coend}
\DeclareMathOperator\aut{\underline{aut}}

\DeclareMathOperator{\uaut}{\underline{aut}}
\DeclareMathOperator{\uend}{\underline{end}}

\def\perf{\operatorname{perf}}
\def\ot{\otimes}
\def\opp{\operatorname{op}}
\def\cube{\operatorname{Cube}}

\usepackage{xcolor}
\hypersetup{
    colorlinks,
    linkcolor={red!50!black},
    citecolor={blue!50!black},
    urlcolor={blue!80!black}
}

\mathchardef\mhyphen="2D

\newcounter{todocounter}
\DeclareDocumentCommand\addreference{g}{\stepcounter{todocounter}\todo[color = blue!30, fancyline]{\thetodocounter. Add reference\IfNoValueF{#1}{: #1}}\xspace}
\DeclareDocumentCommand\checkthis{g}{\stepcounter{todocounter}\todo[color = red!50, fancyline]{\thetodocounter. Check this\IfNoValueF{#1}{: #1}}\xspace}
\DeclareDocumentCommand\fixthis{g}{\stepcounter{todocounter}\todo[color = orange!50, fancyline]{\thetodocounter. Fix this\IfNoValueF{#1}{: #1}}\xspace}
\DeclareDocumentCommand\expand{g}{\stepcounter{todocounter}\todo[color = green!50, fancyline]{\thetodocounter. Expand\IfNoValueF{#1}{: #1}}\xspace}

\title[The Manin Hopf algebra]
{The Manin Hopf algebra of a Koszul Artin-Schelter regular algebra is quasi-hereditary}
\author{Theo Raedschelders and Michel Van den Bergh}
\address{(Theo Raedschelders) \newline Departement Wiskunde, Vrije Universiteit Brussel, 
Pleinlaan $2$, B-1050 Elsene \newline E-mail address: {\tt traedsch@vub.ac.be}}
\address{(Michel Van den Bergh) \newline Departement WNI, Universiteit Hasselt, Universitaire Campus \\
B-3590 Diepenbeek \newline E-mail address: {\tt michel.vandenbergh@uhasselt.be}}
\begin{document}

\begin{abstract}
  For any Koszul Artin-Schelter regular algebra $A$, we consider a
  version of the universal Hopf algebra $\uaut(A)$ coacting on $A$,
  introduced by Manin. To study the representations (i.e. finite
  dimensional comodules) of this Hopf algebra, we use the
  Tannaka-Krein formalism. Specifically, we construct an explicit
  combinatorial rigid monoidal category $\Uscr$, equipped with a
  functor~$M$ to finite dimensional vector spaces such that $\uaut(A)
  = \coend_\Uscr(M)$. Using this pair $(\Uscr,M)$ we show that
  $\uaut(A)$ is quasi-hereditary as a coalgebra and in addition is
  derived equivalent to the representation category of $\Uscr$.
\end{abstract}

\maketitle

\setcounter{tocdepth}{1}
{\hypersetup{linkcolor=black}
\tableofcontents
}

\newcommand\blfootnote[1]{%
  \begingroup
  \renewcommand\thefootnote{}\footnote{#1}%
  \addtocounter{footnote}{-1}%
  \endgroup
}

\blfootnote{The first author is an aspirant at the FWO.} 
\blfootnote{The second author is a senior researcher at the FWO.}

\section{Introduction}

\subsection{Motivation}
Let $k$ be an algebraically closed field. 

In his beautiful notes on non-commutative geometry \cite{manin}  Manin
constructs for any graded algebra $A=k\oplus A_1\oplus A_2\oplus\cdots$ 
   a bialgebra
$\uend(A)$ and a Hopf algebra $\uaut(A)$ coacting on it in a universal way.
The Hopf algebra $\uaut(A)$ should be thought of as the non-commutative
symmetry group of $A$.

Both $\uend(A)$ and $\uaut(A)$ are  large objects. Typically they are of exponential growth even if $A$ has polynomial growth.
An unbiased observer might regard this as a good thing. It simply means that $A$ has a large number of non-commutative symmetries.

Still, it is always tempting to look for smaller quotients of $\uaut(A)$  more amenable to 
commutative intuition. Hopf algebras arising in this way are 
traditional quantum groups such as $\Oscr_q(\GL_d)$ which are deformations
of coordinate rings of algebraic groups.  But
such nice quotients 
rarely exist.  For instance no traditional quantum group
can coact on a 3-dimensional elliptic Sklyanin algebra~\cite{Ohn}, a basic object in non-commutative algebraic geometry
\cite{AS}. 
On the other hand, despite its size $\uend(A)$ is
 reasonable whenever $A$ is. For example when~$A$ is Koszul, a complete description of the  representation
theory  of $\uend(A)$ (as a coalgebra) was obtained in~\cite{kriegk-vandenbergh}. So there are no issues with $\uend(A)$, even if $A$
is a Sklyanin algebra. 

In this paper we will revisit the  work in \cite{kriegk-vandenbergh} and fully extend it to 
 $\uaut(A)$. This provides yet more evidence that $\uend(A)$ and $\uaut(A)$ are well behaving objects which should be studied on their own terms.
However $\uaut(A)$, which is simply the Hopf envelope of $\uend(A)$, turns out to be a much more complicated object than $\uend(A)$, making the generalisation highly non-trivial.
To get a sense of this one may try first to describe $\uaut(A)$ by explicit generators and relations in the most basic
case $A=k[x_1,\ldots,x_d]$. Consulting Appendix \ref{ref-A-156} will
likely convince the reader that working with explicit equations is not a good way to proceed.
Indeed in this paper we completely avoid explicit equations by relying on the Tannakian formalism instead. This will be explained below in \S\ref{ref-1.3-9}, 
after we have stated our main results.
\subsection{Main results}
As is often the case in non-commutative algebraic geometry it seems that some type of Gorenstein property is necessary for $\uaut(A)$ to be well behaved.
Therefore, in this paper, we assume throughout that $A$ is a Koszul Artin-Schelter regular algebra  \cite{AS}\footnote{AS-regular algebras are also known as graded twisted Calabi-Yau algebras.}
of global dimension $d$. 
Thus among other things~$A$ is a finitely presented quadratic graded algebra  $TV/(R)$ such that the  minimal resolution
of $k$ has the form
\[
0\r A\otimes R_d\r \cdots \r A\otimes R_l\r \cdots \r A\otimes R \r A\otimes V\r A\r k\r 0
\]
with\footnote{We usually omit tensor product signs.} $
R_l:=\bigcap_{i+j+2=l}V^i R V^j $. In particular we have $R_2=R$ and
for uniformity we also put $R_1=V$. It follows from the basic
properties of AS-regular algebras that $\dim R_d=1$ and that moreover
the obvious inclusions $R_d\hookrightarrow R_a R_{d-a}$ define
non-degenerate pairings between $R_a$ and $R_{d-a}$. These properties
characterise the AS-regular algebras among the Koszul
ones \cite{smith}. 

For a coalgebra $C$ let $\comod(C)$ be its category of finite dimensional comodules, also called representations.
It is easy to see that the $(R_l)_l$ are representations of $\uaut(A)$, with $R_d$ being invertible.  We may use the $(R_l)_l$ as building blocks to construct more complicated representations. 
To formalise this let $\Lambda^+$ be the free monoid $\langle r_1,\ldots,r_d\rangle$ on the variables $r_1,\ldots,r_d$ and similarly let 
$\Lambda=\langle r_1,\ldots,r_{d-1}, r_d^{\pm 1}\rangle$. We put $M(r_i)=R_i$ and
we use this to define $M(\lambda)$ for arbitrary $\lambda\in \Lambda$ by sending products to tensor products.

Next we equip $\Lambda$ with the left and right invariant partial ordering generated by $r_{a+b}<r_a r_b$, $1<r_ar_d^{-1}r_{d-a}$. We also equip $\Lambda$ with two order preserving dualities $(-)^\ast$
and ${}^\ast(-)$ 
respectively determined by $r_a^\ast=r_{d-a}r_d^{-1}$ and ${}^\ast r_a=r_d^{-1}r_{d-a}$. 
Put
\begin{equation}
\label{ref-1.1-0}
\begin{aligned}
	\nabla(\lambda) &= \coker \bigg( \bigoplus_{\substack{M(\mu)\r M(\lambda)
\text{ in $\comod(\uaut(A))$}
\\ \mu < \lambda}} M(\mu) \to M(\lambda) \bigg)\\
\Delta(\lambda) &= \ker \bigg( M(\lambda) \to \bigoplus_{\substack{M(\lambda)\r M(\mu) \text{ in $\comod(\uaut(A))$}\\ \mu < \lambda}} M(\mu) \bigg).
\end{aligned}
\end{equation}
The following theorem is our first main result.
\begin{theorems}[Corollary \ref{ref-6.3.2-134}, Theorem \ref{ref-7.1.3-147}]
\label{ref-1.2.1-1}
The coalgebra $\uaut(A)$ is a quasi-hereditary coalgebra in the sense of Donkin~\cite{donkin} 
with standard representations $\Delta(\lambda)_{\lambda\in \Lambda}$ and costandard 
representations $\nabla(\lambda)_{\lambda\in \Lambda}$.
\end{theorems}
We refer the reader not familiar with quasi-hereditary algebras to 
\S\ref{ref-2.10-45},\S\ref{ref-2.11-55} 
below for a short introduction. Here we will be content with noting
that the fact that $\uaut(A)$ is quasi-hereditary immediately implies that it has a large
number of standard properties. See Proposition \ref{ref-2.10.8-53} for a non-exhaustive list. For example, for every $\lambda\in
\Lambda$, there is, up to scalar, a unique non-zero morphism
$\Delta(\lambda)\r \nabla(\lambda)$ and its image $L(\lambda)$ is a
simple $\uaut(A)$-representation. Furthermore all the simple $\uaut(A)$-representations
are of this form. The representations
$\Delta(\lambda)$, $\nabla(\lambda)$ have simple top and socle
respectively and their other composition factors $L(\mu)$ satisfy
$\mu<\lambda$. In fact they are maximal with respect to these
properties.

Of course the reader might object that definition \eqref{ref-1.1-0}
is not very useful since it refers to the very object $\uaut(A)$ we
are studying. In fact the actual definition we use in the body of the
paper does not have this defect. It will be given in~\eqref{ref-1.3-8} below
after we have introduced some more notation.

It is easy to see that the (co)standard representations are compatible with the dualities on $\Lambda$ as follows: 
\[
M(\lambda)^\ast=M(\lambda^\ast), \quad \nabla(\lambda)^\ast=\Delta(\lambda^\ast), \quad L(\lambda)^\ast=L(\lambda^\ast)
\]
and a corresponding statement for left dualities. 
 We also have
\[
M(r_a)=\Delta(r_a)=\nabla(r_a)=L(r_a).
\]
In particular the $R_a$ are simple $\uaut(A)$-representations. 
Moreover we have short exact sequences  (see \eqref{ref-6.10-109}) 
\[
0\r \nabla(w\wedge w')\r \nabla(w)\otimes \nabla(w')\r \nabla(ww')\r 0
\]
where $\wedge$ is some simple operation on $\Lambda$.

Denote by $\Fscr(\Delta)$ (respectively $\Fscr(\nabla)$) the categories of $\uaut(A)$-comodules that have a $\Delta$-filtration (respectively $\nabla$-filtration).
\begin{corollarys}
\begin{enumerate}
\item \label{ref-1-2} $\Fscr(\Delta)$ and $\Fscr(\nabla)$ are closed under tensor products. 
\item \label{ref-2-3} $M(\lambda)\in \Fscr(\Delta)\cap \Fscr(\nabla)$.
\end{enumerate}
\end{corollarys}
Property \eqref{ref-2-3} combined with the theory of quasi-hereditary (co)algebras
(see \S\ref{ref-2.10-45}, \S\ref{ref-2.11-55}) shows that the
  $M(\lambda)$ are partial tilting modules (but they are not
  indecomposable). Property~\eqref{ref-1-2} is not a formal
  consequence of the quasi-hereditary property, which in any case is not concerned with the monoidal structure,
but it also holds for example for representations of
  algebraic groups \cite{donkin3,mathieu}.

In general the representation theory of $\uaut(A)$ is very similar to that of algebraic groups. For example 
we have the following result.
\begin{propositions}[Corollary \ref{ref-6.3.3-135}]
 \label{ref-1.2.3-4} 
The representation ring of $\uaut(A)$
 is given by
\[
\ZZ\langle r_1,\ldots,r_{d-1},r_d,r_d^{-1}\rangle
\]
where $r_i$ corresponds to $[R_i]$.
\end{propositions}
The reader should compare this to the fact that in any characteristic the representation ring of $\GL_d$ is of the form $\ZZ[r_1,\ldots,r_{d-1},r_d,r_d^{-1}]$.

The following result clarifies the relation between $\uend(A)$ and $\uaut(A)$.
\begin{propositions}[Corollary \ref{ref-6.3.5-143}]
\label{ref-1.2.4-5}
$\uend(A)$ is the minimal subcoalgebra of $\uaut(A)$ whose representations
have simple composition factors in the set $\{L(\lambda)_{\lambda\in \Lambda^+}\}$.
\end{propositions}
So the relation
between $\comod(\uaut(A))$ and $\comod(\uend(A))$ is similar to the one that exists between the representations of $\GL_n$ and its
full monoidal subcategory of polynomial representations. 

A subset $\pi \subset \Lambda$ is called saturated if $\mu \leq \lambda \in \pi$ implies $\mu \in \pi$. 
As $\Lambda^+$ is a saturated subposet of $\Lambda$ it follows from Proposition \ref{ref-1.2.4-5}
and the theory of quasi-hereditary coalgebras (\S\ref{ref-2.10-45},\S\ref{ref-2.11-55})
that $\uend(A)$ has an induced quasi-hereditary
structure. In fact, as was  mentioned, the category $\comod(\uend(A))$ has already
been completely described in~\cite{kriegk-vandenbergh}. Specifically it is a
direct sum of representation categories of explicit directed quivers. Our next aim will be to state a similar
result for $\comod(\uaut(A))$. It will however not take place on the abelian level but on the derived level.

First we upgrade $\Lambda$ to a monoidal category defined by an explicit presentation. More specifically we introduce a rigid monoidal category $\Uscr$ such that $\Ob(\Uscr)=\Lambda$
with generating morphisms $\phi_{a,b}:r_{a+b}\r r_ar_b$, $\theta_{a,b}:r_{a}r^{-1}_d r_{b}\r r_{a+b-d}$ (with $r_0=1$), satisfying
certain natural relations 
given in \S\ref{ref-3-62} below. 
The category $\Uscr$ turns out to be highly structured. In particular it has a triangular decomposition: the classes of morphisms $(\phi_{a,b})_{a,b}$, $(\theta_{a,b})_{a,b}$ generate two 
directed Koszul subcategories $\Uscr_{\u}$, $\Uscr_{\d}$ of~$\Uscr$
such that every morphism $f$ in $\Uscr$ can be uniquely written as\footnote{Technically $\Uscr$ is an example of a Reedy category. See ~\cite{riehl-verity}.}  $f_{\u}\circ f_{\d}$ with $f_{\u}\in \Uscr_{\u}$, $f_{\d}\in \Uscr_{\d}$.
This makes it in particular possible to explicitly compute the $\Hom$-sets in~$\Uscr$. It is also easy to see that $M$ may be upgraded to a functor $\Uscr\r \comod(\uaut(A))$.

We define the category $\Mod(\Uscr^\circ)$ of right $\Uscr$-modules as the category of
contravariant functors from $\Uscr$ to vector spaces. Let $\perf(\Uscr^\circ)$
be the triangulated category of finite complexes of finitely generated
projective right $\Uscr$-modules. One may make $\perf(\Uscr^\circ)$ into a monoidal
triangulated category by putting $k\Uscr(-,\lambda)\otimes
k\Uscr(-,\mu)=k\Uscr(-,\lambda\mu)$. We have the following result.
\begin{theorems}[Theorem \ref{ref-7.1.3-147} and Corollary \ref{ref-7.2.1-151}] 
\label{ref-1.2.5-6}
The monoidal functor
\[
M:k\Uscr\r \comod(\uaut(A)):\lambda\mapsto M(\lambda)
\]
is  fully faithful and its derived version
\begin{equation}
\label{ref-1.2-7}
M:\perf(\Uscr^\circ)\r D^b(\comod(\uaut(A)))
\end{equation}
induced by
$
k\Uscr(-,\lambda)\mapsto M(\lambda)
$
is an equivalence of monoidal triangulated categories.
\end{theorems}
As is explained in \S\ref{ref-7-144} the first claim may be
regarded as a form of Schur-Weyl duality, whereas the second claim is a derived 
version of the classical Tannaka-Krein duality for abelian categories
with faithful fiber functors (see Theorem \ref{ref-2.9.3-41}). Theorem~\ref{ref-1.2.5-6} may also be interpreted as saying that $k\Uscr$ is Morita equivalent to (a suitably defined categorical  version of) the Ringel dual (see \S\ref{ref-2.10-45}) of $\uaut(A)$.
See Remark \ref{ref-7.2.2-152}.

The equivalence \eqref{ref-1.2-7} may be used to transfer the standard
$t$-structure on the derived category $D^b(\comod(\uaut(A)))$ to one on $\perf(\Uscr^\circ)$. In \S\ref{ref-7.2-149} below we 
are able to give an intrinsic description of this induced $t$-structure referring solely to  properties of $\Uscr$. As a corollary we obtain.
\begin{theorems}[Theorem \ref{ref-7.2.3-153}] The category
  $\comod(\uaut(A))$ as a monoidal category only depends on the global
  dimension of $A$ and not on $A$ itself.
\end{theorems}
In other words by letting $A$ vary we obtain non-trivial examples of Morita equivalent Hopf algebras \cite[\S5]{schauenburg}.

Let us close this section by giving better formulas for the standard and costandard representations which were promised above. From
Theorem \ref{ref-1.2.5-6}(1) we deduce that \eqref{ref-1.1-0} is equivalent to:
\begin{equation}
\label{ref-1.3-8} 
\begin{aligned}
	\nabla(\lambda) &= \coker \bigg( \bigoplus_{\substack{\mu \to \lambda \in\Uscr\\ \mu < \lambda}} M(\mu) \to M(\lambda) \bigg)\\
	\Delta(\lambda) &= \ker \bigg( M(\lambda) \to \bigoplus_{\substack{\lambda \to \mu\in \Uscr \\ \mu < \lambda}} M(\mu) \bigg).
\end{aligned}
\end{equation}
These formulas no longer refer to $\uaut(A)$.
\subsection{The Tannaka-Krein formalism}
\label{ref-1.3-9}
For a more elaborate discussion of the Tannakian formalism see \S\ref{ref-2.9-39} below. Let $\Cscr$ be a category and
let $F:\Cscr\r \Vect$ be a functor to finite dimensional vector spaces. We refer to such a functor as a fiber functor. The natural transformations
from $F$ to itself form a  $k$-algebra $\End(F)$, equipped with a natural pseudo-compact topology and its continuous dual is 
a coalgebra denoted by
$\coend_\Cscr(F)$. There is an obvious evaluation functor
\[
\operatorname{ev}_F:\Cscr\r \comod(\coend_\Cscr(F)):X\mapsto FX
\]
If $C$ an arbitrary coalgebra equipped with additional
algebraic structure then this is 
reflected into the categorical properties of $\comod(C)$. For example if $C$
is equipped with a bialgebra structure then $\comod(C)$ is a monoidal category.
 The Tannakian
formalism goes in some sense the other way. It connects categorical structures on $\Cscr$ to algebraic
structures on $\coend_\Cscr(F)$.
In particular we have
the following properties~\cite{pareigis-1}
\begin{itemize}
\item If $\Cscr$ is monoidal and $F$ is a monoidal functor\footnote{We always assume that for a monoidal functor  the coherence maps $FA \otimes FB \to F(A \otimes B)$ and $1 \to F1$ are invertible.}
then $\coend_\Cscr(F)$ is a bialgebra and $\eval_F$ is a monoidal functor.  
\item  If in addition $\Cscr$ has right duals, then $\coend_\Cscr(F)$ is a 
	Hopf algebra.
\item If in addition $\Cscr$ is rigid, then $\coend_\Cscr(F)$ is a Hopf algebra with invertible antipode.
\end{itemize}
Moreover we have compatibility with duals \cite[Theorem 4.5]{pareigis-1}:
\begin{itemize}
\item Let $\Cscr$ be a monoidal category, let $\Cscr^\ast$ be obtained from $\Cscr$ by adjoining right duals and let $F^\ast:\Cscr^\ast\r\Vect$ be the corresponding monoidal extension of~$F$. Then
$\coend_{\Cscr^\ast}(F^\ast)$ is the Hopf envelope of $\coend_\Cscr(F)$.
\end{itemize}
The key observation in this paper is that $\uend(A)$ and $\uaut(A)$ can
be realised as $\coend_\Cscr(F)$ for suitably chosen nice fiber functors
and that moreover such realisations can be used to prove
properties about them. 

Let $\Cscr$ be the monoidal category generated by objects $r_1$,
$r_2$ and a single morphism $\phi_{11}:r_2\r r_1r_1$ and let $F$ be
the fiber functor given by $F(r_i)=R_i$ and $F(\phi_{11})$ being the
inclusion $R_2\hookrightarrow R_1R_1$. Then it is easy to see that
$\uend(A)=\coend_{\Cscr}(F)$. However this is not yet  what we want
as the pair $(\Cscr,F)$ insufficiently captures the structure of $\comod(\uend(A))$.  However starting
from $(\Cscr,F)$ 
and some algebraic manipulations we obtain 
\begin{theorems}[Theorem \ref{ref-5.1-92}] 
Let $\Uscr_{\u}^+$ be the full subcategory of $\Uscr$ with objects~$\Lambda^+$. Let $M^+$ be
the restriction of $M$ to $\Uscr_{\u}^+$. Then
\begin{align*}
		\uend(A) & \cong \coend_{\Uscr_{\u}^+}(M^+), \\
		\uaut(A) & \cong \coend_{\Uscr}(M).
\end{align*}
\end{theorems}
Now assume that $\Lambda_1$, $\Lambda_2$ are saturated subsets 
of $\Lambda$ such that the elements of $\Lambda_2-\Lambda_1$ are incomparable.
Let $\Uscr_i\subset \Uscr$ be the full subcategories of $\Uscr$ with object sets $\Lambda_i$. 
The key technical result that enters in the proof of Theorem \ref{ref-1.2.1-1} is the following.
\begin{theorems}[Theorem \ref{ref-6.2.3-120}] \label{ref-1.3.2-10}
There is an exact sequence
\begin{equation}
\label{ref-1.4-11}
	0 \to \prod_{\lambda\in\Lambda_2-\Lambda_1}\Hom_k(\nabla(\lambda),\Delta(\lambda)) \to \End_{\Uscr_2}(M) \to \End_{\Uscr_1}(M) \to 0,
\end{equation}
 where $\nabla(\lambda),\Delta(\lambda)$ are as in \eqref{ref-1.3-8} above.
\end{theorems}
Starting with \eqref{ref-1.4-11} we may construct
a heredity cochain in $\coend_{\Uscr}(M)$ (see \S\ref{ref-2.11-55}) which
yields that $\coend_{\Uscr}(M)$ is quasi-hereditary.

\subsection{Relation with other work}
Various types of universal quantum groups have been introduced in the literature.
E.g.\ \cite{bichon,bichon-dubois-violette,chirvasitu,walton-wang}.  See 
in particular \cite[\S2]{walton-wang} for a very nice survey. These papers either consider quotients of $\uaut(A)$
or else discuss different or restricted settings. 

In   \cite{pareigis-1} Pareigis shows how to use the Tannakian formalism to
obtain presentations for $\uaut(A)$. An application of this idea is given
in Appendix \ref{ref-A-156}. However as was already pointed out, we do 
not use explicit presentations in this paper. Instead we use
the Tannakian formalism directly to obtain strong results on the
representation theory of $\uaut(A)$. We believe this idea is new.

In the case that $A=k[x_1,\ldots,x_d]$ the properties of $\uend(A)$
and $\uaut(A)$ are strongly connected to those of so-called Manin matrices
(Appendix \ref{ref-A-156}). For an extensive survey on Manin-matrices see
\cite{chevrov-falqui-rubtsov}.  For applications to integrable systems
see \cite{rubtsov-silantyev-talalaev}.
The bialgebra $\uend(A)$  was used in \cite{hai-lorenz} to obtain a vast
generalisation of the classical MacMahon master theorem in combinatorics.

In our companion paper \cite{raedschelders-vandenbergh} we give a more in-depth
discussion of $\aut(k[x_1,x_2])$.  The methods in loc.\ cit.\ are more
akin to those used in the theory of algebraic groups. 
In particular 
 we construct the standard and
costandard representations using induction from a Borel subalgebra and we use
highest weight theory to construct the simple representations. Such methods are outside the scope of the current paper.

\section{Preliminaries}

\subsection{Notation and conventions}

Let~$k$ denote an algebraically closed field. By convention $\otimes=\otimes_k$, and often we will omit the tensor sign altogether, for example~$V \otimes_k V=V \otimes V=VV=V^2$. Also, $\VVect$ denotes the category of vector spaces over~$k$, and $\Vect$ is its full subcategory of finite dimensional vector spaces. Similarly, we use $\VVect_{\ZZ}$ and $\Vect_{\ZZ}$ for $\ZZ$-graded vector spaces.

We will need to consider the strict monoidal categories generated by a
number of objects $\{X_i\}_i$, morphisms $\{f_j\}_j$ and relations
$\{\phi_k\}_k$. We will use the notation~$\langle \{X_i\}_i |
\{f_j\}_j | \{\phi_k\}_k \rangle$ for such a category. The details of
this construction are spelled out in Section~3.1 of~\cite{pareigis-1}.

We often pass from a coalgebra $C$ to its dual algebra $C^*$. In the case of infinite dimensional coalgebras, we freely use the duality between coalgebras and pseudocompact algebras on the one hand, and the equivalence between $C$-comodules and discrete $C^*$-modules on the other. By $\comod(C)$ we denote the category of finite dimensional $C$-comodules. We will refer to an object in $\comod(C)$ as a `representation' of $C$.

\subsection{Modules over linear categories}

We briefly review some familiar notions from ring theory in the context of $k$-linear categories, viewed as `rings with several objects'. This analogy has been fully developed in~\cite{mitchell}. 

Let $\mathcal{C}$ denote a small $k$-linear category. A left
$\mathcal{C}$-module is then a covariant linear functor $F:\mathcal{C}
\to \VVect$. These functors form a Grothendieck category $\Mod
\mathcal{C}$. We will also write $\mod \Cscr\subset \Mod\Cscr$ for the full subcategory spanned by finitely
generated modules. In our applications
$\mod \Cscr$ will always be an abelian category.

If $\mathcal{D}$ is another small $k$-linear category,
then a $\mathcal{C}$-$\mathcal{D}$-bimodule is a linear covariant
functor $\mathcal{C} \to \Mod \mathcal{D^{\opp}}$, or equivalently a
$k$-linear bifunctor $\mathcal{D}^{\opp} \times \mathcal{C} \to
\VVect$. If $F:\Cscr \to \Dscr$ is a $k$-linear functor then we will
regard it as a $\Cscr$-$\Dscr$-bimodule by putting
$F(D,C)=\Dscr(D,F(C))$. The $\Cscr$-$\Cscr$-bimodule corresponding to
the identity functor is $\Cscr(-,-)$ (or informally $\Cscr$ itself).

\subsection{Distributive lattices of vector spaces}
\label{ref-2.3-12}
In Section~\ref{ref-6-101}, we will need the following simple properties of lattices of vector spaces.

	\begin{propositions}\cite[Ch. 1, \S 7]{polishchuk-positselski}
	\label{ref-2.3.1-13}
	Suppose a collection of subspaces $V_1, \ldots, V_n$ of a vector space $V$ generates a distributive lattice. 
	Then:
		\begin{enumerate}
		\item there is a basis $\{v_{i}\}_i$ of $V$ such that each of the subspaces $V_i$ is the linear span of a set of
		vectors $v_i$,
		\item the dual collection $V_1^{\perp}, \ldots, V_n^{\perp} \subset V^*$ also generates a distributive lattice,
		\item for any other distributive collection $V'_1, \ldots, V'_n \subset V'$ the tensor product collection 
		$V_1 \ot V'_1, \ldots, V_n \otimes V'_n \subset V \ot V'$ is distributive.
		\end{enumerate}
	\end{propositions}
	
	\begin{propositions}\cite[Ch. 1, Proposition 7.2]{polishchuk-positselski}
	\label{ref-2.3.2-14}
	Given a collection of subspaces $V_1, \ldots, V_n \subset V$ such that any proper subset 
	$V_1, \ldots, \widehat{V_k}, \ldots, V_n$ generates a distributive lattice. Then the following are equivalent:
	\begin{enumerate}
	\item the collection $V_1, \ldots, V_n$ generates a distributive lattice,
	\item the complex of vector spaces 
		\begin{equation}
		0 \to \bigcap_s V_s \to \cdots \to \bigoplus_{t_1 < \cdots < t_{n-i}} \bigcap_{s=1}^{n-i} V_{t_i} \to \cdots \to 
		\bigoplus_t V_t \to V
		\end{equation}
		is exact everywhere except for the rightmost term.
	\end{enumerate}
	\end{propositions}

In \cite{polishchuk-positselski} Polishchuk and Positselski considered 
the category $Q_n$ which is  an
$n$-dimensional hypercube with commuting faces. Objects in $Q_n$ are subsets $I\subset \{1,\ldots,n\}$ with  arrows
$I\r J$ for $I\subset J$. 
\begin{theorem} \label{ref-2.1-15} \cite[Ch. 2, Lemma 9.1]{polishchuk-positselski}. 
Let $V_1, \ldots, V_n \subset V$ be a collection of subspaces and consider
the functor
\[
F:Q_n^\circ\to \VVect: I\mapsto \cap_{i\in I} V_i.
\]
The collection $(V_i)_i$  generates a distributive lattice if and only if $F$ can be extended to an
exact functor
\[
F:\mod(Q_n)\to \VVect
\]
which maps the projective object $Q_n(I,-)$ in $\mod(Q_n)$ to $F(I)$.
\end{theorem}

\subsection{Koszul categories and algebras}

Let ${{\Cscr}}$ be a graded $k$-linear category, i.e.\ a category enriched in $\VVect_{\ZZ}$. We let ${{\Cscr}}_n$ be 
the morphisms of degree $n$ in ${{\Cscr}}$. In this way ${{\Cscr}}_0$ is a $k$-linear category.
We say that ${{\Cscr}}$ is connected if ${{\Cscr}}_n$ contains no morphisms for $n<0$ and only scalar multiples 
of the identity morphisms for $n=0$. From now on we assume ${{\Cscr}}$ is connected, ${{\Cscr}}_n(x,y)$ is finite-dimensional for all $n \in \ZZ$ and $x,y \in {{\Cscr}}$, and for any $x,y \in {{\Cscr}}$, there are only finitely many $z \in {{\Cscr}}$ that satisfy both ${{\Cscr}}(x,z) \neq 0$ and ${{\Cscr}}(z,y) \neq 0$. We let $\Gr({{\Cscr}})$ be the category of graded left ${{\Cscr}}$-modules, in other words: covariant graded functors ${{\Cscr}} \r \VVect_{\ZZ}$.
For $M\in \Gr({{\Cscr}})$ and $n\in \ZZ$ we define $M(n)\in \Gr({{\Cscr}})$ by $M(n)(x)_m=M(x)_{m+n}$.
Then $\Gr({{\Cscr}})$ has a system of projective generators given by $({{\Cscr}}(x,-)(n))_{x\in\Ob({{\Cscr}}),n\in \ZZ}$. 

Let $S_x\in \Gr({{\Cscr}})$ for $x\in \Ob({{\Cscr}})$ be defined by 
\begin{equation}
\label{ref-2.2-16}
S_x(y)=
\begin{cases}
k&\text{$x=y$}\\
0&\text{otherwise}
\end{cases}
\end{equation}
It is clear that $S_x$ is a simple object in $\Gr({{\Cscr}})$. 
	\begin{definitions}
	\label{ref-2.4.1-17}
	  The category ${{\Cscr}}$ is Koszul if it is connected and for each $x\in
	  \Ob({{\Cscr}})$ the object $S_x$ has a ``linear'' projective resolution,
	  i.e. a resolution of the form
		\begin{equation}
		\label{ref-2.3-18}
 		\cdots \r  \bigoplus_{i\in I_2}  {{\Cscr}}(x_{2,i},-)(-2) \r \bigoplus_{i\in I_1} {{\Cscr}}(x_{1,i},-)(-1)\r {{\Cscr}}(x,-)\r S_x\r 0
		\end{equation}
	\end{definitions}
Note that we do not assume that the $I_i$ are finite. 
	\begin{remarks}
	It is easy to see that the resolution \eqref{ref-2.3-18}, if it exists, must be unique, up to unique
	isomorphism.
	\end{remarks}
If ${{\Cscr}}$ is Koszul then ${{\Cscr}}$ is generated by
${{\Cscr}}_1$ over ${{\Cscr}}_0$ and moreover ${{\Cscr}}$ is quadratic in the sense that the
canonical map
\[
T_{{{\Cscr}}_0}({{\Cscr}}_1)/(R)\r {{\Cscr}}
\]
with 
\[
R=\ker({{\Cscr}}_1\otimes_{{{\Cscr}}_0}{{\Cscr}}_1\xrightarrow{f\otimes g\r fg} {{\Cscr}}_2)
\]
is an isomorphism, where 
	\begin{equation}
	T_{{{\Cscr}}_0}({{\Cscr}}_1)(-,-)={{\Cscr}}_0(-,-) \oplus {{\Cscr}}_1(-,-) \oplus \big( {{\Cscr}}_1(-,-) \otimes_{{{\Cscr}}_0} {{\Cscr}}_1(-,-) \big) \oplus \cdots.
	\end{equation}
is the free tensor category. We recall some standard facts which are proved exactly as in the ring case.
	\begin{lemmas}
	${{\Cscr}}$ is Koszul if and only if ${{\Cscr}}^\circ$ is Koszul.
	\end{lemmas}
	\begin{proof} We may consider $S_x$ also as a graded right
	  ${{\Cscr}}$-module. Then ${{\Cscr}}$ is Koszul if and only if
	  $\Tor_{{{\Cscr}}}^i(S_x,S_y)$ lives purely in degree $-i$ for all $x,y\in
	  \Ob({{\Cscr}})$. This characterisation is obviously left right symmetric.
	\end{proof}
	
	\begin{lemmas}\cite[Ch.1, \S 6,7 and Ch.2, \S 3,4]{polishchuk-positselski} 
	\label{ref-2.4.4-19}
	Let ${{\Cscr}}=T_{{{\Cscr}}_0}({{\Cscr}}_1)/(R)$ be a connected quadratic category.
	Then ${{\Cscr}}$ is Koszul if and only if the ${{\Cscr}}_0$-bimodules
	\[
	R_{ij}={{\Cscr}}_1^{\otimes_{{{\Cscr}}_0}i}\otimes_{{{\Cscr}}_0} R\otimes_{{{\Cscr}}_0} {{\Cscr}}_1^{\otimes_{{{\Cscr}}_0}j},
	\]
	for $i+j+2=n$, generate a distributive lattice in ${{\Cscr}}_1^{\otimes_{{{\Cscr}}_0} n}$.

	Moreover if we put
	\[
	R_n=\bigcap_{i+j+2=n} R_{ij}
	\]
	and $R_1={{\Cscr}}_1$, $R_0={{\Cscr}}_0$
	then the resolution \eqref{ref-2.3-18} is of the form 
		\begin{multline}
		\label{ref-2.5-20}
		\cdots \r \bigoplus_y {{\Cscr}}(y,-)\otimes_k R_n(x,y) \r \cdots \r \bigoplus_y {{\Cscr}}(y,-)\otimes_k R_2(x,y)  \\
		\r \bigoplus_y {{\Cscr}}(y,-)\otimes_k {{\Cscr}}_1(x,y) \r {{\Cscr}}(x,-)\r S_x\r 0
		\end{multline}
	where the differentials are given by the composition
	\begin{multline*}
	\bigoplus_y {{\Cscr}}(y,-)\otimes_k R_n(x,y)\r \bigoplus_{y,z} {{\Cscr}}(y,-)\otimes_k {{\Cscr}}_1(z,y)\otimes_k R_{n-1}(x,z) \\
	\r \bigoplus_z {{\Cscr}}(z,-)\otimes_k R_{n-1}(x,z)
	\end{multline*}
	\end{lemmas}
	
	This complex will be called the Koszul complex of $x$ and will be denoted $K_{\bullet}({{\Cscr}})(x)$.
	
	\begin{remarks}
\label{ref-2.4.5-21}
	If ${{\Cscr}}$ has only a single object $x$, ${{\Cscr}}(x,x)\cong A$ with ${{\Cscr}}_0(x,x) \cong k$ and ${{\Cscr}}_1(x,x) \cong V$, then Definition~\ref{ref-2.4.1-17} reduces to the classical definition of Koszulity for the connected graded $k$-algebra $A$. Also, the complex~\eqref{ref-2.5-20} coincides with the classical Koszul complex
		$$
		K_{\bullet}(A): \cdots \to A \otimes_k A_n^{!*} \to \cdots \to A \otimes_k A_2^{!*} \to A \otimes_k A_1^{!*} \to A 
		\to k \to 0,
		$$
	since 
		\begin{align}
		A_n^{!*} = \bigg(V^{* \otimes n} / \sum_{i+j+2=n} R_{ij}^{\perp}\bigg)^* 
		= \bigg(\sum_{i+j+2=n} R_{ij}^{\perp}\bigg)^{\perp} 
		= \bigcap_{i+j+2=n} R_{ij}
		=R_n.
		\end{align}
	\end{remarks}

\subsection{Quadratic categories, confluence and Koszulity}

Below we assume that~${{\Cscr}}$ is a connected graded category given by
a homogeneous presentation $T_{{{\Cscr}}_0}({{\Cscr}}_1)/(R)$ where 
$R\subset {{\Cscr}}_1\otimes_{{{\Cscr}}_0}{{\Cscr}}_1$. We assume that we are given
 bases $(f_{x,y,i})_{i\in I_{x,y}}$  for ${{\Cscr}}_1(x,y)$ and we let $F$ be the 
union of all these bases.

Assume furthermore that $F$ is equipped with a total ordering and use this
to equip the composable words in $F$ (which form a basis for $T_{{{\Cscr}}_0}({{\Cscr}}_1)$)
with the lexicographic ordering.
Assume that
$R$ has a basis given by
\begin{equation}
\label{ref-2.7-22}
fg=r 
\end{equation}
where $f,g \in F$ and $r$ is a linear combination of quadratic words in $F$ which
are strictly smaller than $fg$.
\begin{lemmas}\cite{berger}
\label{ref-2.5.1-23}
Assume that the basis elements in $R$ are confluent (the overlaps
of length 3 resolve correctly). Then:
\begin{enumerate}
\item
${{\Cscr}}$ has a basis consisting
of compositions $f_1\cdots f_n$  in $F$ such that no pair $f_if_{i+1}$
occurs as a lefthandside in \eqref{ref-2.7-22}.
\item
${{\Cscr}}$ is Koszul.  Moreover
any sequence of composable arrows $f_1\cdots f_n$ in $F$
where each pair $f_if_{i+1}$ occurs as lefthand side of \eqref{ref-2.7-22}
can be uniquely completed to an element  $f_1\cdots f_n-g$
of $R_n$ where $g$ is a linear combination  of compositions of element in $F$ which are
strictly smaller than $f_1\cdots f_n$. Moreover the elements
of $R_n$ so obtained form a basis for $R_n$.
\end{enumerate}
\end{lemmas}

\subsection{Koszul algebras and monoidal categories}
\label{ref-2.6-24}
Here we recall some results from \cite{kriegk-vandenbergh,polishchuk-positselski} rephrased in the
language we will use in the current paper.
Let $\Lambda^+_{\infty}=\langle r_1,\ldots,r_i,r_{i+1},\ldots\rangle$ be the 
free monoid with infinitely many generators
and let $\Uscr_{\u,\infty}^+$ be the strict monoidal categories derived
from $\Lambda^+_{\infty}$ as follows:
\begin{enumerate}
\item We freely adjoin morphisms $\phi_{a,b}:r_{a+b}\r r_ar_b$.
\item We impose the relations
\begin{equation}
\label{ref-2.8-25}
\xymatrix{
r_{a+b+c}\ar[d]_{\phi_{a+b,c}}\ar[rr]^{\phi_{a,b+c}} && r_ar_{b+c}\ar[d]^{r_a\phi_{b,c}}\\
r_{a+b}r_{c}\ar[rr]_{\phi_{a,b}r_c} && r_ar_{b}r_c
}
\end{equation}
\end{enumerate}
writing $u$ for $\Id_u$ and suppressing tensor products as usual. We will also define a second operation $\wedge$ on $\Uscr_{\u,\infty}^+$ by putting
for $w=r_{i_1}\cdots r_{i_m}$, $w'= r_{j_1}\cdots r_{j_m}$
\[
w\wedge w'
:=r_{i_1}\cdots r_{i_m+j_1}\cdots r_{j_m}
\]
We declare $1\wedge w$ and $w\wedge 1$ to be undefined. There is an obvious canonical map
\[
\phi_{w,w'}:w\wedge w'\r ww'
\]
in $\Uscr_{\u,\infty}^+$ obtained from $\phi_{i_m,j_1}$.
The following is a direct corollary of results in  \cite{polishchuk-positselski}.
\begin{theorems} \begin{enumerate}
\item
The monoidal structure  on $\Uscr^+_{\u,\infty}$ may be extended to a biexact monoidal structure 
on $\mod(\Uscr^{+}_{\u,\infty})$ via 
\begin{equation}
\label{ref-2.9-26}
\Uscr^{+}_{\u,\infty}(-,u)
\otimes \Uscr^{+}_{\u,\infty}(-,v)=\Uscr^{+}_{\u,\infty}(-,uv)
\end{equation}
\item For $u\in \Ob(\Uscr^{+}_{\u,\infty})$ put $P_u=\Uscr^{+}_{\u,\infty}(-,u)$ and let $S_u$ be the simple quotient of $P_u$ (see \eqref{ref-2.2-16}). Then for $u,v
\neq 1$ there is a commutative diagram in $\mod(\Uscr^{+}_{\u,\infty})$
\begin{equation}
\label{ref-2.10-27}
\xymatrix{
&P_{u\wedge v}\ar@/^2em/[rr]^{P_{\phi_{u,v}}}\ar@{->>}[d]\ar@{^(->}[r] &P_{u}\otimes P_v\ar@{->>}[d]\ar@{=}[r]& P_{uv}\ar@{->>}[d]\\
0\ar[r]&S_{u\wedge v}\ar[r] & S_u\otimes S_v\ar[r]& S_{uv}\ar[r]&0
}
\end{equation}
with the lower row being exact. 
\end{enumerate}
\end{theorems}
\begin{proof}
\begin{enumerate}
\item
Put $Q=\coprod_{n\ge -1} Q_n$ where $Q_n$ for $n\ge 0$ was introduced in \S\ref{ref-2.3-12} and $Q_{-1}:=\{\bullet\}$ is a point.
There is an equivalence of categories
\begin{equation}
\label{ref-2.11-28}
\Uscr_{\u,\infty}^+\cong Q^\circ
\end{equation}
which on objects is given by
\begin{align*}
r_{i_1} \cdots r_{i_k}&\mapsto 
\{1,\ldots,i_1-1,i_1+1, \ldots, i_1+i_2-1,i_1+i_2+1, \ldots, i_1+i_2+\cdots +i_k-1\}\\
& \qquad\qquad\qquad\qquad\qquad\qquad\qquad \in Q_{i_1+i_2+\cdots +i_k-1}\\
1&\mapsto \bullet\in Q_{-1}
\end{align*}
The equivalence \eqref{ref-2.11-28} may be used to make $Q$ into a monoidal
category. This monoidal category was introduced in \cite{polishchuk-positselski}. Explicitly it is given by
\begin{equation}
\label{ref-2.12-29}
\otimes:Q_m \times Q_n \to Q_{m+n+1}:I,J\mapsto I\cup (J+m+1)
\end{equation}
for $m,n\ge 0$. One extends this to $Q$ by declaring $\bullet\in Q_{-1}$
to be a unit object. The operation $\wedge$ on $\Uscr_{\u,\infty}^+$ corresponds to the 
operation
\begin{equation}
\label{ref-2.13-30}
\wedge:Q_m \times Q_n \to Q_{m+n}:I,J\mapsto I\cup (J+m)
\end{equation}
which is undefined of $m=-1$ or $n=-1$.
The category $\mod(Q_n)$ is denoted by $\cube_n$ in \cite[\S2.9]{polishchuk-positselski} and 
it is equipped with a monoidal structure
\[
\otimes:\cube_m\times \cube_n\r \cube_{m+n+1}:
(X_\bullet)\otimes (Y_\bullet)\mapsto (Z_\bullet)
\]
where 
\[
Z_{I\cup (J+m+1)}=Z_{I\cup \{m+1\}\cup (J+m+1)}:=X_I\otimes Y_J
\]
This monoidal structure is obviously biexact and it is easily seen
to coincide with  \eqref{ref-2.9-26}.
\item 
This is implicitly contained in \cite{polishchuk-positselski} and explicitly in \cite[Prop.\ 2.2.1]{kriegk-vandenbergh} after translating to the current setting via \eqref{ref-2.11-28}\eqref{ref-2.12-29}\eqref{ref-2.13-30}.\qed
\end{enumerate}
\def\qed{}
\end{proof}
\begin{theorems} Let $A=TV/(R)$ with $R\subset V\otimes V$ be a Koszul algebra. Let $R_n=\cup_{i+j+2=n} V^iRV^j$, $R_1=V$. Consider the monoidal functor 
\[
M^+:\Uscr^+_{\u,\infty}\r \VVect
\]
which sends $r_a$ to $R_a$ and $\phi_{a,b}$ to the inclusion $R_{a+b}\hookrightarrow R_aR_b$. Then
\begin{enumerate}
\item $M^+$ can be extended to an exact monoidal functor
\[
M^+:\mod(\Uscr^{+,\circ}_{\u,\infty})\r \VVect:\Uscr^+_{\u,\infty}(-,u)\mapsto M^+(u)
\]
\item For $u\in \Uscr_{\u,\infty}^+$ put 
\[
\nabla(u)=\coker\left(\bigoplus_{v\r u\in \Uscr_{\u,\infty}^+, v\neq u}M^+(v)\r M^+(u)\right)
\]
Then for $u,v\neq 1$ we have a commutative diagram 
\begin{equation}
\label{ref-2.14-31}
\xymatrix{
&M^+(u\wedge v)\ar@/^2em/[rr]^{M^+(\phi_{u,v})}\ar@{->>}[d]\ar@{^(->}[r] &M^+(u)\otimes M^+(v)\ar@{->>}[d]\ar@{=}[r]& M^+(uv)\ar@{->>}[d]\\
0\ar[r]&\nabla(u\wedge v)\ar[r] & \nabla(u)\otimes \nabla(v)\ar[r]& \nabla(uv)\ar[r]&0
}
\end{equation}
with the lower row being exact. 
\end{enumerate}
\end{theorems}
\begin{proof}
\begin{enumerate}
\item This is a direct consequence of Theorem \ref{ref-2.1-15}.
\item We have $M^+(u)=M^+(P_u)$, $ \nabla(u)=M^+(S_u) $.  The commutativity of \eqref{ref-2.14-31}
  now follows from \eqref{ref-2.10-27} and the fact that $M^+$ is exact
  and monoidal.\qed
\end{enumerate}
\def\qed{}\end{proof}
\subsection{Artin-Schelter regular algebras}
\label{ref-2.7-32}

Let $A$ denote a finitely generated homogeneous $k$-algebra, i.e. $A$ is of the form $TV/(R)$, where $V$ is a finite dimensional $k$-vector space and $R$ is a subspace contained in $V^{\otimes N}$. In particular, this algebra is graded and $A_0 = k$. 

	\begin{definitions} \cite{AS}
	\label{ref-2.7.1-33}
	The algebra $A$ is Artin-Schelter regular of dimension $d$ if
		\begin{enumerate}
		\item $A$ has finite global dimension $d$,
		\item 	
				\begin{equation*}
				\Ext^i_{A}(k,A) = \left \{ 
					\begin{array}{ll}
						0 & \textrm{if $i \neq d$} \\
						k(l) & \textrm{if $i=d$},
					\end{array} \right.
				\end{equation*}
		\end{enumerate}
	for some $l \in \mathbb{Z}$.
	\end{definitions}
	
	\begin{remarks}
	\begin{enumerate}
	\item One often also requires finite GK-dimension. We will not need it in what follows.
	\item One can define both left and right AS-regular algebras, but it turns out that the definition is left-right symmetric, see~\cite[Proposition 3.6]{he-torrecillas-vanoystaeyen-zhang}. In particular, Definition~\ref{ref-2.7.1-33} is unambiguous.
	\end{enumerate}
	\end{remarks}
	
Recall the following:	
	
	\begin{propositions}\cite[Lemma 1.2]{reyes-rogalski-zhang}
	The algebra $A$ is AS-regular of dimension $d$ if and only if it is graded skew Calabi-Yau of dimension $d$, that is
		\begin{enumerate}
		\item $A$ is homologically smooth of dimension $d$, i.e. $A$ is perfect and has projective dimension $d$ as a $k$-linear bimodule over $A$,
		\item $A$ is rigid Gorenstein, i.e. there is a graded algebra automorphism $\mu:A \to A$ such that
				\begin{equation*}
				\Ext^i_{A^e}(A,A^e) = \left \{ 
					\begin{array}{ll}
						{}^{\mu} A^{\id}(l) & \textrm{if $i=d$} \\
						0 & \textrm{if $i \neq d$}
					\end{array} \right.
				\end{equation*}
			as graded $A$-bimodules,
		\end{enumerate}
	for some $l \in \mathbb{Z}$.
	\end{propositions}

\subsection{Universal coacting bialgebras and Hopf algebras}
\label{ref-2.8-34}

Throughout 
	\begin{equation}
	A=k\oplus A_1\oplus A_2\oplus \cdots
	\end{equation}
is an $\mathbb{N}$-graded algebra such that $\dim A_i<\infty$ for all $i$.
We first introduce the universal coacting
bialgebra $\uend(A)$ which is defined using a suitable universal property. Every
bialgebra has a universal enveloping Hopf algebra, which in the case
of $\uend(A)$ will be denoted $\uaut(A)$. This Hopf algebra also
satisfies a universal property and is in fact the universal coacting
Hopf algebra of $A$.

	\begin{definitions}
	\label{ref-2.8.1-35}
	The universal coacting algebra of $A$, denoted $\uend(A)$, is an algebra equipped with an 	
	algebra morphism $\delta_A:A \to \uend(A) \ot A$, such that $\delta_A(A_n)\subset \uend(A)\ot A_n$ satisfying the following universal property: 
	for any algebra $B$ and algebra morphism $f: A \to B \ot A$, such that 
	$f(A_n) \subset B \ot A_n$, there exists a unique morphism 
	$g:\uend(A) \to B$ such that the diagram
	$$
	\begin{tikzcd}
		A \drar[swap]{f} \rar{\delta_A} & \uend(A) \ot A \dar[dashed]{g \ot 1}\\
		& B \ot A 
	\end{tikzcd}
	$$
	commutes.
	\end{definitions}

The existence of this algebra is essentially due to Manin~\cite{manin}. In the following proposition we collect some properties of this algebra.

	\begin{definitions} Let $B$ be a bialgebra. A $B$-comodule algebra is an algebra $A$ equipped with
	an algebra morphism $f:A\rightarrow B\otimes A$ which makes $A$ into a comodule over $B$.
	\end{definitions}
	
	\begin{propositions}\cite[Proposition 1.3.8]{pareigis-2}
	\label{ref-2.8.3-36}
	\begin{enumerate}
	\item The universal coacting algebra of $A$ is a bialgebra,
	\item $A$ is an $\uend(A)$-comodule 
	algebra via $\delta_A$,
	\item $\uend(A)$ also satisfies a different universal property:
	if $B$ is any bialgebra, and $f:A \to B \otimes A$ equips $A$ with the structure
	of a $B$-comodule algebra such that $f(A_n)\subset B\otimes A_n$, then there is a unique morphism of bialgebras $g:\uend(A) \to B$ such 
	that the diagram
	$$
	\begin{tikzcd}
		A \drar[swap]{f} \rar{\delta_A} & \uend(A) \otimes A \dar[dashed]{g \otimes 1} \\
		& B \otimes A
	\end{tikzcd}
	$$
	commutes.
	\end{enumerate}
	\end{propositions}
	\begin{propositions}\cite[Theorem 2.6.3]{pareigis-2}
	\label{ref-2.8.4-37}
	Let $B$ be a bialgebra. Then there exists a Hopf algebra $H(B)$, called the Hopf envelope of $B$, 
	and a homomorphism of bialgebras $i:B \to H(B)$ such that for every Hopf algebra $H$ and for every 
	homomorphism of bialgebras $f:B \to H$, there is a unique homomorphism of Hopf algebras 
	$g:H(B) \to H$ such that the diagram
	$$
	\begin{tikzcd}
		B \drar[swap]{f} \rar{i} & H(B) \dar[dashed]{g} \\
		& H 
	\end{tikzcd}
	$$
	commutes.
	\end{propositions}
        We will denote the Hopf envelope of $\uend(A)$ by
        $\uaut(A)$. Using Definition~\ref{ref-2.8.1-35}, there is a
        morphism of algebras $\delta_A:A \to \uaut(A) \otimes A$ such
        that $A$ is a comodule-algebra over $\uaut(A)$. This easily
        gives the final universal property.

	\begin{corollarys}
	\label{ref-2.8.5-38}
	If $H$ is a Hopf algebra and $A$ is an $H$-comodule algebra by $f:A \to H \otimes A$ such that $f(A_n)\subset H\otimes A_n$, then there is a 
	unique morphism of Hopf algebras $g:\uaut(A) \to H$ such that the diagram 
	$$
	\begin{tikzcd}
	A \drar[swap]{f} \rar{\delta_A} & \uaut(A) \otimes A \dar[dashed]{g \otimes 1} \\
	& H \otimes A
	\end{tikzcd}
	$$
	commutes.
	\end{corollarys}
	\begin{proof}
	First use the universal property of Proposition~\ref{ref-2.8.3-36} to get a morphism 
	$g':\uend(A) \to H$, and then use the one of Proposition~\ref{ref-2.8.4-37} to get a map $g$.
	\end{proof}
Following this corollary we call  $\uaut(A)$ the universal
        coacting Hopf algebra of $A$.
\subsection{Tannaka-Krein formalism}
\label{ref-2.9-39}

In this section we describe how the universal coacting bialgebras and Hopf algebras can be seen from the viewpoint of Tannakian duality theory. 

Let $\Cscr$ denote a category and $F:\Cscr \to \Vect$ a functor. Below we refer to such a functor as a fiber functor. The natural transformations of $F$ form an algebra, denoted by $\End(F)$. This algebra has a natural pseudocompact topology: the base open sets are given by 
	\begin{equation}
\label{ref-2.16-40}
	\cap_{X \in \alpha} \Ker(\End(F) \xrightarrow{\eval_X} \End(FX))
	\end{equation}
for all finite subsets $\alpha \subset \Ob(\Cscr)$. For a review of pseudo-compact $k$-algebras see \cite[\S4]{vandenbergh}.

	\begin{definitions}
	The coalgebra $\coend(F)$ is the continuous dual of the pseudocompact algebra $\End(F)$.
	\end{definitions}

	\begin{remarks}
	If we want to emphasise the domain category of $F$, we will also write $\coend_\Cscr(F)$.
	\end{remarks}

For $X \in \Cscr$, it is obvious that $FX$ is a finite-dimensional left $\End(F)$-module, so $FX$ is a left $\coend(F)$-comodule. This association defines an evaluation functor
	\begin{equation}
	\eval_F: \Cscr \to \comod(\coend(F)): X \mapsto FX.
	\end{equation}
Classical Tannaka-Krein theory is concerned with when this functor is an equivalence of categories, and there is the following fundamental theorem.

	\begin{theorems}\cite[Ch.IV, \S 4]{gabriel}
	\label{ref-2.9.3-41}
	If $\Cscr$ is a $k$-linear abelian category and $F:\Cscr \to \Vect$ is an exact and faithful functor, then the evaluation functor $\eval_F$ is an equivalence of categories.
	\end{theorems}
	
The general philosophy of Tannakian duality theory is that additional categorical structures on $\Cscr$ give additional algebraic structures on the coalgebra $\coend(F)$. 

	\begin{propositions}\cite{pareigis-1} 
	\label{ref-2.9.4-42}
	If the category $\Cscr$ is monoidal and $F$ is a monoidal functor, then $\coend(F)$ is a bialgebra and $\eval_F$ is a monoidal functor. If $\Cscr$ has right duals, then $\coend(F)$ is a 
	Hopf algebra. If $\Cscr$ is rigid, then $\coend(F)$ is a Hopf algebra with invertible antipode.
	\end{propositions}
	\begin{remarks}
	Note that a monoidal functor automatically preserves duals, so this does not have to be imposed on $F$ in Proposition \ref{ref-2.9.4-42}.
	\end{remarks}

Now for every monoidal category there is a universal way of adjoining duals to it.

	\begin{lemmas}\cite[Lemma 4.2]{pareigis-1}
	\label{ref-2.9.6-43}
	Let~$\Cscr$ be a small monoidal category. Then there exists a unique small monoidal category 
	$\Cscr^*$ admitting right duals and a monoidal functor $*:\Cscr \to \Cscr^*$ such that for any small monoidal
	category $\Dscr$ admitting right duals and monoidal functor $F$, there exists a unique monoidal functor $F^*$ making the 
	diagram
	\begin{equation}
	\begin{tikzcd}
	\Cscr \rar{*} \dar[swap]{F} & \Cscr^* \dlar[dashed]{F^*} \\
	\mathcal{D}
	\end{tikzcd}
	\end{equation}
	commute.
	\end{lemmas}
	
The passage from $\Cscr$ to $\Cscr^*$ is compatible with the algebraic structures on $\coend(F)$.

	\begin{propositions}\cite[Theorem 4.5]{pareigis-1}
	\label{ref-2.9.7-44}
	For a monoidal category $\Cscr$ and a monoidal functor $F:\Cscr \to \Vect$, there is an isomorphism of Hopf algebras
		\begin{equation}
		H(\coend(F)) \cong \coend(F^*).
		\end{equation}
	\end{propositions}

\subsection{Finite dimensional quasi-hereditary algebras}
\label{ref-2.10-45}
We first recall the ring-theoretical definition of quasi-hereditary algebras
\cite{dlab-ringel-1}.
In this section $A$ is always a finite dimensional $k$-algebra with Jacobson radical~$\rad(A)$. 

\begin{definitions}
\label{ref-2.10.1-46}
A two-sided ideal~$I$ of~$A$ is called a heredity ideal if
\begin{enumerate}
\item $I$ is idempotent, i.e. $I=I^2$,
\item $I_{A}$ is projective,
\item $I \rad(A) I=0$.
\end{enumerate} 
\end{definitions}
We now state some properties of heredity ideals, see~\cite[Appendix C]{deng-du-parshall-wang}.

	\begin{lemmas}
	Assume $I$ is a heredity ideal in $A$. Then:
	\begin{enumerate}
	\item $I$ is also projective as a left $A$-module. Thus the notion of heredity ideal is left right symmetric.
	\item The obvious graded morphism 
		\begin{equation*}
		\Ext^{\bullet}_{A/I}(M,N) \to \Ext^{\bullet}_{A}(M,N)
		\end{equation*}
	is an isomorphism for $M,N \in \mod-A/I$. 
	\item The algebra $A$ has finite global dimension iff $A/I$ has finite global 
	dimension.
	\end{enumerate}
	\end{lemmas}
	\begin{definitions}[Ring theoretical definition]
	\label{ref-2.10.3-47}
	The algebra~$A$ is quasi-hereditary if it has a filtration by heredity ideals, i.e. there is a chain
	\begin{equation}
	0=J_0 \subset J_1 \subset \cdots \subset J_{m-1} \subset J_m=A
	\end{equation}
	of ideals of~$A$ such that for any~$1 \leq t \leq m$,~$J_t/J_{t-1}$ is a heredity ideal in $A/J_{t-1}$. Such a chain is called a heredity chain.
	\end{definitions}
        Next we discuss the module theoretical definition of
        quasi-hereditary algebras.  let $\{L(\lambda) \ \vert \
        \lambda \in \Lambda \}$ be a complete set of non-isomorphic
        simple $A$-modules for some partially ordered set
        $(\Lambda,\leq)$.  A subset $\pi \subset \Lambda$ is called
        saturated if $\mu \leq \lambda \in \pi$ implies $\mu \in
        \pi$. For $\lambda \in \Lambda$, put $\pi(\lambda)=\{\mu \in
        \Lambda \ \vert \ \mu < \lambda\}$. This is a saturated subset
        of~$\Lambda$. Let~$V$ be an $A$-representation.  For $\pi
        \subset \Lambda$ we say that $V$ belongs to $\pi$ if all
        composition factors of $V$ are in the set $\{L(\lambda) \
        \vert \ \lambda \in \pi \}$. We write $O_\pi(V)$ for the
        module that is maximal amongst all submodules of $V$ belonging
        to $\pi$. Similar we write $O^\pi(V)$ for the minimal
        submodule of $V$ such that $V/O^\pi(V)$ belongs to $\pi$.  By
        $I(\lambda)$, $P(\lambda)$  we denote respectively the injective hull
and the projective cover of the simple module
        $L(\lambda)$. Then $\nabla(\lambda) \supset L(\lambda)$ is the
        submodule of $I(\lambda)$ defined by
	\begin{equation}
	\label{ref-2.21-48}
	\nabla(\lambda)/L(\lambda)=O_{\pi(\lambda)}(I(\lambda)/L(\lambda))
	\end{equation}
The $\nabla(\lambda)$ are called costandard modules. The standard modules $\Delta(\lambda)$ are defined dually as
	$$
	\Delta(\lambda)=P(\lambda)/O^{\pi(\lambda)}(\rad P(\lambda))
	$$
Denote by $\Fscr(\Delta)$ (respectively $\Fscr(\nabla)$) the categories of $A$-modules that have a $\Delta$-filtration (respectively $\nabla$-filtration).
	\begin{definitions}[Module theoretical definition]\cite{donkin}
	\label{ref-2.10.4-49}
	The algebra $A$ is quasi-hereditary 
with respect to the poset $(\Lambda,\le)$ if
	\begin{enumerate}
		\item $I(\lambda) \in \Fscr(\nabla)$,
		\item $[I(\lambda):\nabla(\lambda)]=1$,
		\item If $[I(\lambda):\nabla(\mu)] \neq 0$, then $\mu \geq \lambda$.
	\end{enumerate}
	\end{definitions}
\begin{remarks}
The definition can be equivalently stated in terms
of standard modules and projectives. For the generalisation to infinite dimensional coalgebras it is convenient to use costandard modules.
\end{remarks}
\begin{propositions}\cite[Theorem 3.6]{cline-parshall-scott}
\label{ref-2.10.6-50}
The ring theoretical and module theoretical definitions are equivalent.
\end{propositions}
For the benefit of reader we indicate how one may go from a heredity chain
to a partial ordering and back.
Assume first that $A$ is equipped with a heredity chain 
\[
	0=J_0 \subset J_1 \subset \cdots \subset J_{m-1} \subset J_m=A
\]
 Suppose
$\{L(\lambda) | \lambda \in \Lambda\}$ is a (not yet ordered)
labelling of the non-isomorphic simple~$A$-modules. Let $a(\lambda)$
be the maximal index $j$ such that $J_jL(\lambda)=0$. We put
$\lambda>\mu$ iff $a(\lambda)<a(\mu)$. One may show that $A$ is quasi-hereditary in the sense of Definition~\ref{ref-2.10.4-49} with this ordering on $\Lambda$.
With this choice of ordering the standard modules are given by the 
indecomposable summands of $J_{i+1}/J_i$ and the costandard modules
are given by the indecomposable summands of $(J_{i+1}/J_i)^\ast$.

Conversely assume that $A$ is quasi-hereditary with respect to the poset $(\Lambda,\le)$. We may refine $\le$ to a total ordering $\lambda_1>\lambda_2>\cdots$.
We put $J_0=0$ and $J_i=O^{\pi(\lambda_i)}(A)$. One checks this yields a heredity chain.

The reader will note that stating that an algebra is quasi-hereditary
involves specifying some extra data (for example a partial ordering on the
simples or a heredity chain). Two such quasi-hereditary structures are 
considered equivalent if they yield the same collection of standard and
costandard modules. This may also be expressed by using the notion
of a minimal partial ordering.
Specifically, if~$A$ is a quasi-hereditary algebra with poset
$(\Lambda,\le)$ then there is a unique minimal partial ordering
$\le_{\min}$ on $\Lambda$ such that $(A,\Lambda,\le_{\min})$ is
quasi-hereditary with the same standard and costandard comodules as
$(A,\Lambda,\le)$.  The minimal partial ordering is the minimal
ordering on $\Lambda$ with the following property (see
\cite[\S1]{parshall-scott-wang}):
\begin{equation}
\label{ref-2.22-51}
[\Delta(\mu):L(\lambda)]\neq 0\text{ or }[\nabla(\mu):L(\lambda)]\neq 0\Rightarrow \lambda\le_{\min} \mu
\end{equation}
\begin{lemmas} 
\label{ref-2.10.7-52} Assume that $A$ is a quasi-hereditary algebra. Let
$<$ be some partial order on $\Lambda$. Then $<$ is finer than the minimal ordering if for every total ordering $\lambda_1<'\lambda_2<'\cdots$ refining $<$, $A$
is quasi-hereditary with respect to this ordering with the original standard and costandard comodules. 
\end{lemmas}
\begin{proof} Assume $[\Delta(\mu):L(\lambda)]\neq 0$ or $[\nabla(\mu):L(\lambda)]\neq 0$. If $\lambda\not<\mu$ then we may choose a total ordering such that $\mu<'\lambda$. But this is in contradiction with the fact that we must have $\lambda<'\mu$ since $A$ is quasi-hereditary with respect to $<'$.
\end{proof}
For further reference we remind the reader of some standard properties
of quasi-hereditary algebras. 
\begin{propositions}\cite{dlab-ringel-2}
\label{ref-2.10.8-53}
Assume that $A$ is quasi-hereditary with respect to $(\Lambda,\le)$.
\begin{enumerate}
\item $\Delta(\lambda)$, $\nabla(\lambda)$ are Schurian.
\item One has
$
\Ext^i(\Delta(\lambda),\nabla(\mu))=k^{\delta_{i0}\delta_{\lambda\mu}}
$.
\item The image of the non-zero map $\Delta(\lambda)\r\nabla(\lambda)$
(unique up to a scalar by (2)) is equal to $L(\lambda)$.
\item If $\Ext^\ast(\Delta(\lambda),\Delta(\mu))\neq 0$ then $\lambda\le\mu$.
\item If $\Ext^\ast(\nabla(\lambda),\nabla(\mu))\neq 0$ then $\lambda\ge\mu$.
\end{enumerate}
\end{propositions}
The objects in the additive category  $\Fscr(\Delta)\cap \Fscr(\nabla)$ 
are particularly interesting since by Proposition \ref{ref-2.10.8-53}(2) they satisfy
$\Ext^i(-,-)=0$ for $i>0$. So in particular they are (partial) tilting modules.
\begin{propositions}\cite{dlab-ringel-2}
\label{ref-2.10.9-54} Assume that $A$ is
  quasi-hereditary with respect to $(\Lambda,\le)$. For every
  $\lambda$ there exists a unique indecomposable $A$-module in
  $\Fscr(\Delta)\cap \Fscr(\nabla)$ such that
  $[T(\lambda):\Delta(\lambda)]=[T(\lambda):\nabla(\lambda)]=1$ and if
  $[T(\lambda):\Delta(\mu)]\neq 0$ or $[T(\lambda):\nabla(\mu)]\neq 0$
  then $\mu\le \lambda$. Moreover $\Fscr(\Delta)\cap
  \Fscr(\nabla)=\add \{T(\lambda)_{\lambda}\}$.
\end{propositions}
If we put $T=\bigoplus_{\lambda\in \Lambda} T(\lambda)$ then $\End(T)$ is the so-called Ringel dual of $A$. It is again a quasi-hereditary algebra.

\subsection{Quasi-hereditary coalgebras}
\label{ref-2.11-55}
If $A$ is a finite dimensional algebra then $A^\ast$ is a coalgebra
and moreover $\mod(A)\cong \comod(A^\ast)$. Furthermore any coalgebra is
locally finite: it is the union of its finite dimensional
subcoalgebras. So it should not be surprising that
concepts defined
for finite dimensional algebras can often be generalised to arbitrary
coalgebras. The theory of quasi-hereditary algebras is no exception.

Both the ring theoretical and the module theoretical definition may be generalised to infinite dimensional coalgebras
	\begin{definitions}[(Co)ring theoretical definition]
	\label{ref-2.11.1-56}
	A (possibly infinite dimensional) coalgebra $C$ is quasi-hereditary if there exists an exhaustive filtration
		\begin{equation}
		\label{ref-2.23-57}
		0 \subset C_1 \subset C_2 \subset \cdots \subset C_n \subset \cdots
		\end{equation}
	of finite dimensional subcoalgebras such that for every $i$,
		\begin{equation}
		0=(C_i/C_i)^* \subset (C_i/C_{i-1})^* \subset (C_i/C_{i-2})^* \subset \cdots \subset C_i^*
		\end{equation}
	is a heredity chain. Such a filtration is called a heredity cochain.
	\end{definitions}

        The following module-theoretical definition is due to
        Donkin~\cite{donkin}. Assume~$C$ is a (possibly infinite
        dimensional) coalgebra and let $\{L(\lambda) \ \vert \ \lambda
        \in \Lambda \}$ be a complete set of non-isomorphic simple
        $C$-comodules for some partially ordered set
        $(\Lambda,\leq)$. 

For $\pi$ a (not necessarily finite) saturated subset of $\Lambda$ we let
 $C(\pi):=O_\pi C$, which is in
        fact a coalgebra by maximality.

	\begin{definitions}[(Co)module theoretical definition] \cite{donkin}
	\label{ref-2.11.2-58}
	The coalgebra $C$ is (Donkin) quasi-hereditary if there is a (possibly infinite) poset $(\Lambda, \leq)$ indexing non-isomorphic simple comodules such that
		\begin{enumerate}
			\item for every $\lambda\in \Lambda$ the set $\pi(\lambda)$ is finite;
			\item for every finite, saturated $\pi \subset \Lambda$, the coalgebra $C(\pi)$ is finite
  			dimensional and $C(\pi)^*$ is a quasi-hereditary algebra for the partially ordered set of simples
$(L(\lambda)_{\lambda\in\pi}$.
		\end{enumerate}
	\end{definitions}

To check compatibility with finite, saturated subsets one uses the following theorem. For $\pi \subset \Lambda$ and $\lambda \in \pi$ we write $\Delta_{\pi}(\lambda)$ and $\nabla_{\pi}(\lambda)$ for the corresponding $C(\pi)$-(co)standard comodules. 

	\begin{theorems}  \label{ref-2.11.3-59} \cite[Prop. A.3.4]{donkin2}
	Assume that $C$ is a finite-dimensional quasi-hereditary coalgebra. For a saturated subset $\pi\subset \Lambda$ we have
	that $C(\pi)$ is quasi-hereditary with simple, standard and costandard comodules respectively
	given by $L(\lambda)$, $\Delta_\pi(\lambda)=\Delta(\lambda)$, $\nabla_\pi(\lambda)=\nabla(\lambda)$ for $\lambda\in \pi$.
	\end{theorems}

Now assume that $C$ is infinite dimensional and Donkin quasi-hereditary. For $\lambda\in \Lambda$ and $\pi$ a saturated subset in $\Lambda$ containing $\lambda$ (e.g.\ $\pi(\lambda)$)
we put
\begin{align*}
\Delta(\lambda)&=\Delta_{\pi}(\lambda)\\
\nabla(\lambda)&=\nabla_{\pi}(\lambda)\,.
\end{align*}
Theorem \ref{ref-2.11.3-59} shows that this definition is independent
of $\pi$, so (co)standard comodules also make sense for
infinite-dimensional coalgebras.  In the same manner $L(\lambda)$, $T(\lambda)$ (see Proposition \ref{ref-2.10.9-54}) are
independent of the particular finite dimensional coalgebra $C(\pi)$ for which they were defined.

The following
theorem by Donkin \cite[Thm 2.5]{donkin}
shows that in general the homological algebra of quasi-hereditary
coalgebras is completely determined by that of their finite dimensional
quasi-hereditary subcoalgebras. 
	\begin{theorems}
\label{ref-2.11.4-60}
	If $C$ is quasi-hereditary, then for a finite, saturated $\pi \subset \Lambda$, and $C(\pi)$-comodules $V$ and $W$, one has for all $i \geq 0$,
	\[
	\operatorname{Ext}_{C(\pi)}^i(V,W) \cong \operatorname{Ext}_C^i(V,W).
	\]
	\end{theorems}

	\begin{propositions}
	\label{ref-2.11.5-61}
	If a coalgebra $C$ is quasi-hereditary, then it is Donkin quasi-hereditary. Moreover the converse is true if $\Lambda$ is countable. 
	\end{propositions}
	\begin{proof}
This is proved by reduction to the finite dimensional case and
invoking Proposition \ref{ref-2.10.6-50}. For the converse one notes that
any partial ordering on a countable set may be refined to a total ordering
of the form $\lambda_1<\lambda_2<\cdots$.
	\end{proof}
        Since in the sequel we only use countable partially ordered
        sets we will make no distinction between the two notions of
        quasi-hereditary. Moreover we will freely transfer terminology and results from the
        finite dimensional algebra case to the coalgebra case.
For further reference we also note:
\begin{propositions}
Theorems \ref{ref-2.11.3-59} and \ref{ref-2.11.4-60} remain valid for $\pi$ infinite.
\end{propositions}
\begin{proof} This follows by bootstrapping from the case $|\pi|<\infty$.
\end{proof}
\section{Constructing the combinatorial category $\Uscr$}
\label{ref-3-62}
We want to incorporate $\uaut(A)$ in the Tannakian duality framework
by constructing a strict, rigid monoidal category $\Uscr$ and a
monoidal functor $M:\Uscr \to \Vect$ in such a way that $\coend(M)
\cong \uaut(A)$ as Hopf algebras. Using some version of 
Theorem~\ref{ref-2.9.3-41} one could take
$\Uscr=\comod(\uaut(A))$ and $M$ the forgetful functor, but this does
not help us in understanding the representations themselves. The goal
therefore is to construct a couple $(\Uscr,M)$ which is sufficiently
combinatorial to work with, yet carries a lot of information about the
representations of $\uaut(A)$.

\subsection{The category $\Uscr_{\u}$}
Since by our standing hypothesis $A$ is a Koszul Artin-Schelter regular algebra of global dimension $d$, the minimal resolution
of $k$ is ``capped'' in degree $d$. Thus by Remark \ref{ref-2.4.5-21} one has that $R_n=0$ for $n>d$. This suggests
that in the monoidal category $\Uscr_{\u,\infty}^+$ introduced in \S\ref{ref-2.6-24} we should only consider the
full monoidal subcategory generated by the objects $(r_i)_{1\le i\le d}$. This is the base idea for what we discuss below.

Let $\Lambda^+$ be the monoid $\langle r_1,\ldots,r_d\rangle$ and let
$\Lambda$ be obtained from $\Lambda^+$ by adjoining an inverse
$r_d^{-1}$ for $r_d$. Thus $\Lambda=\langle
r_1,\ldots,r_{d-1},r_d^{\pm 1}\rangle$. Let $d_{\u}$ be the morphism of
monoids
\[
d_{\u}:\Lambda\r (\ZZ,+):r_i\mapsto 1
\]
If $w\in \Lambda$ then we write $l(w)$ for the length of $w$ written in the shortest
possible way as a word in $r_1,\ldots,r_{d-1},r_d^{\pm 1}$.
We let
$\Uscr_{\u}^+,\Uscr_{\u}$ be the strict monoidal categories obtained from $\Lambda^+,\Lambda$ in the same way as $\Uscr^+_{\u,\infty}$ was obtained from $\Lambda^+_{\infty}$
in \S\ref{ref-2.6-24}. In particular $\Uscr_{\u}^+$ is a full subcategory of $\Uscr_{\u,\infty}^+$ and $\Uscr_{\u}$ is obtained from $\Uscr_{\u}^+$ by universally
inverting $r_d$.

\begin{remarks} While $\Uscr^+_{\u}$ is a poset, this is not the case for $\Uscr_{\u}$. For example if $d=2$  then there
are two distinct morphisms $r_1r_1\r r_1r_1r_2^{-1}r_1r_1$ given respectively by
$r_1r_1r_2^{-1}\phi_{11}$ and $\phi_{11}r_2^{-1}r_1r_1$. 
\end{remarks}
We will now restrict our attention to $\Uscr_{\u}$. Similar
but easier arguments are valid for~$\Uscr^+_{\u}$, $\Uscr^+_{\u,\infty}$.
We make $k\Uscr_{\u}$ into a graded category by declaring an element $f \in \Uscr_{\u}(u,u')$ to be
homogeneous of degree
$
\deg(f)=d_{\u}(u')-d_{\u}(u)
$.
\begin{propositions}
The category $k\Uscr_{\u}$ is Koszul.
\end{propositions}
\begin{proof} 
By definition  $\Uscr_{\u}$ as non-monoidal category is generated by
\[
U_{\u,1}:=\{u\phi_{a,b}v\mid u,v\in \Ob(U_{\u}),1\le a,b\le d, a+b\le d\}
\]
Since $\deg(u\phi_{a,b}v)=1$ it is clear that $k\Uscr_{\u}$ is connected. One checks
that if
\[
u\phi_{a,b}v=u'\phi_{a',b'}v'
\]
then $u=u'$, $v=v'$ and $a=a'$, $b=b'$ so that  $\Uscr_{\u,1}$ forms a basis for $(k\Uscr_{\u})_1$.

The relations of $\Uscr_{\u}$ as non-monoidal category are given by
\begin{equation}
\label{ref-3.1-63}
\xymatrix{
ur_{a+b+c}v\ar[d]_{u\phi_{a+b,c}v}\ar[rr]^{u\phi_{a,b+c}v} && ur_ar_{b+c}v\ar[d]^{ur_a\phi_{b,c}v}\\
ur_{a+b}r_{c}v\ar[rr]_{u\phi_{a,b}r_c v} && ur_ar_{b}r_cv
}
\end{equation}
together with the tautological relations
\begin{equation}
\label{ref-3.2-64}
\xymatrix{
ur_{a+b}vr_{c+e}w\ar[d]_{ur_{a+b}v\phi_{c,e}w}\ar[rr]^{u\phi_{a,b}vr_{c+e}w} && ur_a r_b vr_{c+e}w\ar[d]^{ur_ar_bv\phi_{c,e}w}\\
ur_{a+b} vr_cr_ew \ar[rr]_{u\phi_{a,b}vr_cr_ew}&&ur_ar_bvr_cr_ew
}
\end{equation}
with $u,v\in \Ob(\Uscr_{\u})$. We now choose of a total order on the generators of $\Uscr_{\u}$ in such a way that 
if  $l(u)<l(u')$ then
\begin{equation}
\label{ref-3.3-65}
u\phi_{a,b}v<u'\phi_{a',b'}v',
\end{equation}
This has the effect that if we write the relations on $\Uscr_{\u}$ linearly in the form
\eqref{ref-2.7-22}, the lefthand side will be the upper branch in the commutative diagrams
\eqref{ref-3.1-63}\eqref{ref-3.2-64}.

We  claim that the relations \eqref{ref-3.1-63}\eqref{ref-3.2-64} are confluent. This implies that $k\Uscr_{\u}$ is Koszul by Lemma~\ref{ref-2.5.1-23}. There are
4 possible overlaps to check. We give one example which is the overlap of \eqref{ref-3.2-64} with
itself.
\[
(ur_ar_b \phi_{c,e}v)\circ (ur_a\phi_{b,c+e}v)\circ(u\phi_{a,b+c+e}v)
\]
Rewriting, starting with the first two factors yields
\begin{align*}
&[(ur_ar_b \phi_{c,e}v)\circ (ur_a\phi_{b,c+e}v)]\circ(u\phi_{a,b+c+e}v)\\
&=(ur_a\phi_{b,c}r_ev)\circ [(ur_a\phi_{b+c,e}v)\circ (u\phi_{a,b+c+e}v)]\\
&=[(ur_a\phi_{b,c}r_ev)\circ   (u\phi_{a,b+c}r_ev)]  \circ (u\phi_{a+b+c,e}v)\\
&=(u\phi_{a,b}r_cr_ev)\circ   (u\phi_{a+b,c}r_ev)  \circ (u\phi_{a+b+c,e}v)
\end{align*}
Starting with the last two factors we get
\begin{align*}
&(ur_ar_b \phi_{c,e}v)\circ [(ur_a\phi_{b,c+e}v)\circ(u\phi_{a,b+c+e}v)]\\
&=[(ur_ar_b \phi_{c,e}v)\circ (u\phi_{a,b}r_{c+e}v)]\circ (u\phi_{a+b,c+e}v)\\
&=(u\phi_{a,b} r_{c}r_ev)\circ [(ur_{a+b}\phi_{c,e}v)\circ (u\phi_{a+b,c+e}v)]\\
&=(u\phi_{a,b} r_{c}r_ev)\circ (u\phi_{a+b,c}r_ev)\circ (u\phi_{a+b+c,e}v)\qed
\end{align*}
\def\qed{}\end{proof}

\begin{remarks}
Not every composable pair of generators for $\Uscr_{\u}$ participates in a quadratic relation. For example 
if $d=2$ then the composition 
\[
r_2 \xrightarrow{\phi_{11}} r_1r_1 \xrightarrow{r_1r_2^{-1}\phi_{11}r_1} r_1r_2^{-1}r_1r_1r_1
\]
cannot be written in any other way.
\end{remarks}

\begin{remarks}
Note that $r^p_d\phi_{a,d-a}r_d^{-p-1}$ is a generator starting in $1$ and ending in $r_d^pr_{a}r_{d-a}r_d^{-p-1}$.
So every object in $\Uscr_{\u}$ has infinitely many outgoing morphisms of degree one. On the other hand it is easy to see that if $u'\r u\in \Uscr_{\u}$ is not
the identity then $l(u')<l(u)$. Hence there are only a finite number of
morphisms in $\Uscr_{\u}$ with target a given object $u$.
\end{remarks}
	\begin{lemmas}
	\label{ref-3.1.5-66}
	\begin{enumerate}
	\item
	All maps in $\Uscr_{\u}$ are mono's. In other words if there is a commutative
	diagram in $\Uscr_{\u}$
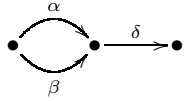
	\[
	\xymatrix@1{
	\bullet \ar@/^1em/[r]^\alpha \ar@/_1em/[r]_\beta & \bullet \ar[r]^\delta & \bullet
	}
	\]
	then $\alpha=\beta$.
	
	\item If there is a commutative diagram in $\Uscr_{\u}$
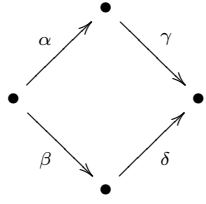
	\begin{equation}
	\label{ref-3.4-67}
	\xymatrix{
	&\bullet\ar[dr]^\gamma\\
	\bullet\ar[ur]^\alpha\ar[dr]_\beta&&\bullet\\
	&\bullet\ar[ur]_\delta
	}
	\end{equation}
	then the fiber product of $\gamma$ and $\delta$ exists.
	
	\item
	If in diagram \eqref{ref-3.4-67} $\gamma$, $\delta\in \Uscr_{\u,1}$ then
	$\gamma$, $\delta$ are not of the form
	
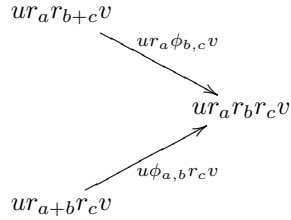
\[
	\xymatrix{
	&ur_ar_{b+c}v\ar[dr]^{ur_a\phi_{b,c}v}\\
	&&ur_ar_br_cv\\
	&ur_{a+b}r_{c}v\ar[ur]_{u\phi_{a,b}r_cv}
	}
	\]
	for $a+b+c>d$ and this is also not true if we exchange $\gamma,\delta$.
	\item If in diagram \eqref{ref-3.4-67} $\gamma$, $\delta\in \Uscr_{\u,1}$
	  then the fiber product is given by the diagrams~\eqref{ref-3.1-63} and~\eqref{ref-3.2-64} (which exist because of (3)).
	\end{enumerate}
	\end{lemmas}

\begin{proof}
One would expect these to be relatively straightforward combinatorial facts. Unfortunately the proof we have is not very direct and proceeds via a double induction.
We first note the following:
\begin{equation}
\label{ref-3.5-68}
\text{To prove (1)(2) we may assume $\gamma\neq \delta\in \Uscr_{\u,1}$.} 
\end{equation}

If $\alpha=u\phi_{a,b}v\in \Uscr_{\u,1}$ for $u,v\in \Ob(\Uscr_{\u})$ then we will put $w(\alpha)=l(u)$. 
For $n\ge 0$ consider the following statement.
	
	\def\Hyp{\operatorname{Hyp}}
	\def\Hp{({\mathversion{bold}$\Hyp_0$})}
	\def\Hypn{({\mathversion{bold}$\Hyp_n$})}
	\def\Hypnn{({\mathversion{bold}$\Hyp_{n+1}$})}
	
	\noindent \Hypn
	\begin{enumerate}
	\item[(a)] (1) holds for $\delta\in \Uscr_{\u,1}$ and
	$\deg(\alpha)+2w(\delta)\le n$ and
	\item[(b)] If we have a diagram like~\eqref{ref-3.4-67} with  $\gamma\neq\delta\in
	\Uscr_{\u,1}$ and  $\deg(\alpha)+w(\gamma)+w(\delta)\le n$ 
	then \eqref{ref-3.4-67} may be completed as 
	
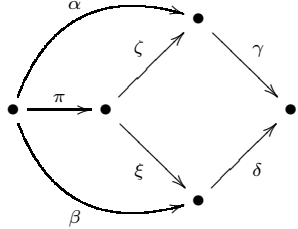
\begin{equation}
	\label{ref-3.6-69}
	\xymatrix{
	&&\bullet\ar[dr]^\gamma\\
	\bullet\ar@/^2em/[urr]^{\alpha}\ar@/_2em/[drr]_{\beta}\ar[r]^\pi &\bullet\ar[ur]^{\zeta}\ar[dr]_{\xi}&&	\bullet\\
	&&\bullet\ar[ur]_\delta
	}
	\end{equation}
	with $\zeta$, $\xi\in \Uscr_{\u,1}$.
	\end{enumerate}
	
  Obviously \Hypn(a) for all $n$ implies (1) for $\delta\in \Uscr_{\u,1}$. Furthermore $\zeta$, $\xi$
  in \Hypn(b) depend only on $\gamma,\delta$ and $\pi$ is unique by (1). Hence
if \Hypn(b) holds for all $n$ then $(\zeta,\xi)$ is the fiber product
of $(\gamma,\delta)$ and from which we easily deduce (2)(3)(4) for
$\gamma,\delta\in \Uscr_{\u}$. We may now conclude by~\eqref{ref-3.5-68}.

It is clear that \Hp\ holds. We will now assume \Hypn\ holds and deduce from it that \Hypnn\ holds.

We start with the proof of \Hypnn(a). We may assume that 
$\alpha$, $\beta$ are written in minimal form as a product of generators according to~\eqref{ref-3.3-65}.

If $\delta\alpha$, $\delta\beta$ are still in minimal form then there is nothing to prove since minimal forms are unique. Hence we may assume that one of them,
say $\delta\beta$ can be further reduced. In other words we may write
$\beta=\beta_1\beta_2$ and
\[
\delta\beta_1=\delta'\beta'_1
\]
with $\delta'\in \Uscr_{\u,1}$ and $w(\delta')<w(\delta)$ (in particular $\delta\neq \delta'$).

We now have 
\[
\delta\alpha=\delta'\beta'_1\beta_2
\]
and
$\deg(\alpha) + w(\delta)+w(\delta')<\deg(\alpha) + 2w(\delta)\le n+1$. Since
$\delta\neq\delta'$ we invoke \Hypn(b) to deduce that there exists an $\epsilon$ such that
$\alpha=\beta_1\epsilon$, $\beta'_1\beta_2=\beta'_1\epsilon$.
Invoking \Hypn(a) we get $\epsilon=\beta_2$ and hence $\alpha=\beta_1\beta_2=\beta$. 

Now we prove \Hypnn(b). We may assume that $\alpha$, $\beta$ are written in minimal form. Since $\gamma\neq \delta$ we may assume that one of the expressions $\gamma\alpha$ or
$\delta\beta$, say the second one,  can be further reduced. In other words we may write
$\beta=\beta_1\beta_2$ and
\[
\delta\beta_1=\delta'\beta'_1
\]
with $\beta_1,\beta'_1,\delta'\in \Uscr_{\u,1}$ and $w(\delta')<w(\delta)$ (in particular $\delta\neq \delta'$). We now have 
\begin{equation}
\label{ref-3.7-70}
\gamma\alpha=\delta'\beta'_1\beta_2
\end{equation}
If $\delta'=\gamma$ then by \Hypnn(a) we find $\alpha=\beta'_1\beta_2$. This
means \Hypnn(b) holds with $\zeta=\beta'_1$, $\xi=\beta_1$, $\pi=\beta_2$.

Hence we reduce to the case $\delta'\neq \gamma$.  Since $\deg(\alpha)
+ w(\gamma)+w(\delta')<\deg(\alpha) + w(\gamma)+w(\delta)\le n+1$ we
deduce from \eqref{ref-3.7-70} and \Hypn(b) that there exist
$\sigma_1,\sigma_2\in \Uscr_{\u,1}$, $\theta\in \Uscr_{\u}$ such that
\begin{align}
\gamma\sigma_1&=\delta'\sigma_2\\
\sigma_1\theta&=\alpha \\
\sigma_2\theta&=\beta'_1\beta_2\label{ref-3.10-71}
\end{align}
Assume $\sigma_2=\beta'_1$. Then we have in $\Uscr_{\u,2}$.
\[
\gamma\sigma_1=\delta'\sigma_2=\delta'\beta'_1=\delta\beta_1
\]
This is impossible by $\delta\neq\delta'\neq\gamma\neq\delta$ since there are at most two ways of writing an element of $\Uscr_{\u,2}$ as a product of elements in $\Uscr_{\u,1}$.

We conclude $\sigma_2\neq\beta'_1$. From \eqref{ref-3.10-71} and \Hypn(b) we deduce the existence of $\sigma_3,\sigma_4\in \Uscr_{\u,1}$, $\theta'\in \Uscr_{\u}$ such that
$\sigma_2\sigma_3=\beta'_1\sigma_4$, 
$\theta=\sigma_3\theta'$ and  $\beta_2=\sigma_4\theta'$.

It is now time to make a diagram of the maps that have been constructed
(plus a few new ones)
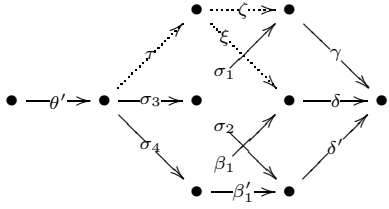
\[
\xymatrix{
&&\bullet\ar@{.>}[r]|{\zeta}\ar@{.>}[dr]|(.3){\xi}&\bullet\ar[dr]|{\gamma}&\\
\bullet\ar[r]|{\theta'}&\bullet\ar@{.>}[ur]|{\tau}\ar[r]|{\sigma_3}\ar[dr]|{\sigma_4}&\bullet\ar[ur]|(.3){\sigma_1}\ar[dr]|(.3){\sigma_2}&\bullet\ar[r]|{\delta}&\bullet\\
&&\bullet\ar[r]|{\beta_1'}\ar[ur]|(0.3){\beta_1}&\bullet \ar[ur]|{\delta'}&
}
\]
All maps in this diagram have degree one except $\theta'$. For the benefit of the reader we restate what the relation with the original maps is
\begin{align*}
\alpha&=\sigma_1\sigma_3\theta'\\
\beta&=\beta_1\sigma_4\theta'
\end{align*}
The existence
of the dotted part of the diagram follows from the combinatorics of arrows in $\Uscr_{\u,1}$. Using the arrows in the dotted part we obtain
\begin{align*}
\alpha&=\sigma_1\sigma_3\theta'=\zeta\tau\theta'\\
\beta&=\beta_1\sigma_4\theta'=\xi\tau\theta'
\end{align*}
Putting $\pi=\tau\theta'$ we have constructed $(\zeta,\xi,\pi)$ as in
\eqref{ref-3.6-69}.
\end{proof}
	
Now denote by $\Uscr_{\u}\,\tilde{}$ the category which is equal to $\Uscr_{\u}$ with an initial object $\ast$ adjoined such that $\Uscr_{\u}\,\tilde{}(u,\ast)=\emptyset$ for $u\in \Ob(\Uscr_{\u})$. 

	\begin{propositions}
\label{ref-3.1.6-72}
	\begin{enumerate} 
	\item All morphisms in $\Uscr_{\u}\,\tilde{}$ are mono,
	\item $\Uscr_{\u}\,\tilde{}$ has fiber products.
	\end{enumerate}
	\end{propositions}
	\begin{proof}
\begin{enumerate}
\item
	The only morphisms added are the $\ast \to u$, which are mono since there are no morphisms ending in $\ast$, so the first claim follows from Lemma~\ref{ref-3.1.5-66}(1).
      \item Let $\gamma, \delta$ be morphisms in $\Uscr_{\u}\,\tilde{}$
        with the same target. If one of them is not in~$\Uscr_{\u}$
        then their fiber product is $\ast$. If they are both in
        $\Uscr_{\u}$ and  they participate in a square like
        \eqref{ref-3.4-67} then their fiber product can be computed in
        $\Uscr_{\u}$ using Lemma~~\ref{ref-3.1.5-66}(2). If they do not participate in such a square then
        their fiber product is $\ast$. \qed
\end{enumerate}	
\def\qed{}\end{proof}
\begin{remarks} 
It is necessary to adjoin an initial object. Assume $d=2$. Then the
maps $r_2 r_1 \r r_1r_1 r_1$ and $r_1r_2\r r_1r_1r_1$ have no fiber-product in $\Uscr_{\u}$ but
in $\Uscr_{\u}\,\tilde{}$ the fiber product is $\ast$.
\end{remarks}

	\begin{corollarys}
	\label{ref-3.1.8-73}
Let $x\in \Ob(\Uscr_\u\,\tilde{}\,)$ and let $(f_i:x_i \to x)_i$ be  morphisms in $\Uscr_{\u}$. Then 
the following complex 	with the standard alternating sign maps is exact in $\Mod((\Uscr_{\u}\,\tilde{}\,)^\circ)$ 
	  \begin{multline}
	  \label{ref-3.11-74}
	\cdots \r\bigoplus_{i_1<\cdots<i_n} k\Uscr_\u\,\tilde{}\,(-,\operatorname{source}(f_{i_1}\times_x \dots \times_x f_{i_n}))\r \cdots \\
	\cdots \r \bigoplus_{i<j}
k\Uscr_\u\,\tilde{}\,(-,\operatorname{source}(f_{i}\times_x f_j))\r \cdots \\
 \r \bigoplus_i 
k\Uscr_\u\,\tilde{}\,(-,x_i)
\r k\Uscr_\u\,\tilde{}\,(-,x)
	\end{multline}
	\end{corollarys}
	\begin{proof}
	This is immediate after evaluating on $z\in \Ob(\Uscr_{\u}\,\tilde{}\,)$ since fiber products become intersections as all the $f_i$ are mono by Proposition~\ref{ref-3.1.6-72}.
	\end{proof}

	\begin{corollarys}
Let $x\in \Ob(\Uscr_{\u}\,\tilde{})$. Then $S_x$ (see \eqref{ref-2.2-16}) has a linear projective resolution in $\Mod((\Uscr_{\u}\,\tilde{}\,)^\circ)$ 
of the form
	\begin{multline}
	\label{ref-3.12-75}
	\cdots \r\bigoplus_{i_1<\cdots<i_n}  k\Uscr_\u\,\tilde{}\,(-,\operatorname{source}(f_{i_1}\times_x \dots \times_x f_{i_n}))\r \cdots \\
	\cdots \r \bigoplus_{i<j}    k\Uscr_\u\,\tilde{}\,(-,\operatorname{source}(f_{i}\times_x f_j))\r  \bigoplus_{f_i:x_i\r x\in \Uscr_{\u,1}} 
k\Uscr_\u\,\tilde{}\,(-,x_i)    
\r  k\Uscr_\u\,\tilde{}\,(-,x)
\r S_x\r 0.
	\end{multline}
	\end{corollarys}
	\begin{proof}
	By Corollary~\ref{ref-3.1.8-73} it is sufficient to verify that the cokernel of the next to last non-trivial map is 
	$S_x$ which is obvious.
	\end{proof}
\begin{remarks}
\label{ref-3.1.10-76}
If $u=r_d^{-a_0} u_1r_d^{-a_1}u_2 r_{d}^{-a_2}\cdots r_d^{-a_{n-1}} u_n r_d^{-a_n}$ with $u_i\in \Lambda^+$, $a_i>0$ then 
$f:u'\r u\in \Uscr_{\u,1}$  is obtained from a morphism $f:u'_i\r u_i$ in $\Uscr^+_{\u}$ for certain $i$. In this way the computation of
the fiber products $f_1\times_x\cdots\times_x f_n$ in \eqref{ref-3.12-75}
reduces to the computation of fiber products in 
$\Uscr^+_{\u}$.
\end{remarks}

\subsection{The category $\Uscr_{\d}$}

Define a morphism $(-)^\ast:\Lambda\r \Lambda^\circ$ of monoids
\[
r_a^\ast=r_{d-a}r_d^{-1}
\]
where here and below we put $r_0=1$. We also define
${}^\ast(-):=((-)^\ast)^{-1}$. 
For $u\in \Lambda$ we put $d_{\d}(u):=-d_{\u}(u^\ast)$. Thus
\[
d_{\d}(r_a)=
\begin{cases}
0&\text{if $a\neq d$}\\
1&\text{if $a=d$}
\end{cases}
\]
 For $u,v\in \Lambda$ we put
\[
\Uscr_{\d}(u^*,v^*):=\Uscr_{\u}(v,u)
\]
If we consider $f\in \Uscr_{\u}(v,u)$ as an element of $\Uscr_{\d}(u^*,v^*)$ then we write it as $f^\ast$.
With this definition $\Uscr_{\d}$ is a category and $(-)^\ast$ defines an isomorphism $\Uscr^\circ_{\u} \cong \Uscr_{\d}$. 
By construction this isomorphism is compatible with the monoid structure on $\Ob(\Uscr_{\d})=\Lambda$ and
it makes $\Uscr_{\d}$ into a monoidal category such that $(uv)^\ast=v^\ast u^\ast$ for $u,v\in \Ob(\Uscr_{\d})$
and the same statement for morphisms.

We make $k\Uscr_{\d}$ into a graded category in the same way as $k\Uscr_{\u}$ but using $d_{\d}$ instead of $d_{\u}$. In this way all results for $\Uscr_{\u}$ may be transferred to $\Uscr_{\d}$.

As a monoidal category $\Uscr_{\d}$ is generated by
$\phi_{a,b}^\ast:r_{d-b}r_d^{-1}r_{d-a}r_d^{-1}\r r_{d-a-b}r_d^{-1}$. It will be 
convenient to use more symmetric generators 
\[
\theta_{a,b}=\phi_{d-b,d-a}^\ast \cdot r_d
\]
which is a morphism
\[
\theta_{a,b}:r_ar_d^{-1}r_b\r r_{a+b-d}
\]
The dual version of \eqref{ref-2.8-25} is
	\begin{equation}
	\label{ref-3.13-77}
	\xymatrix{
	r_ar^{-1}_dr_br_d^{-1}r_c\ar[rr]^{\theta_{a,b}r_d^{-1}r_c} \ar[d]_{r_ar_d^{-1} \theta_{b,c}}&& r_{a+b-d}r_d^{-1}r_c\ar[d]^{\theta_{a+b-d,c}}\\
	r_{a} r_d^{-1} r_{b+c-d}\ar[rr]_{\theta_{a,b+c-d}}&&r_{a+b+c-2d}
	}
	\end{equation}

\subsection{The category $\Uscr$}

Let $\tilde{\Uscr}=\Uscr_{\d}\ast \Uscr_{\u}$ be the category with object set $\Lambda$ and morphisms 
freely generated by the morphisms in $\Uscr_{\d}$ and $\Uscr_{\u}$. $\tilde{\Uscr}$ is strict monoidal
in the obvious way.
Let $\Uscr$ be the
monoidal quotient of $\tilde{\Uscr}$ obtained by imposing by the following relations
	
	\begin{enumerate}
	\item
	\begin{equation}
	\label{ref-3.14-78}
	\xymatrix{
	r_{a+b}r_d^{-1}r_c\ar[rr]^{\phi_{a,b}r_d^{-1}r_c}\ar[d]_{\theta_{a+b,c}}&&r_{a}r_b r_d^{-1} r_c\ar[d]^{r_a\theta_{b,c}}\\
	r_{a+b+c-d}\ar[rr]_{\phi_{a,b+c-d}}&&r_{a} r_{b+c-d}
	}
	\end{equation}
	where $d\le b+c$ and where moreover we allow the degenerate cases $a+b=d$ in which case
	we put $\theta_{d,c}=\Id_{r_c}$ and $b+c=d$ in which case we put $\phi_{a,0}=\Id_{r_a}$.
	\item
	\begin{equation}
	\label{ref-3.15-79}
	\xymatrix{
	r_ar_{d}^{-1} r_{b+c}\ar[rr]^{r_ar_d^{-1}\phi_{b,c}}\ar[d]_{\theta_{a,b+c}}&& r_a r_d^{-1} r_br_c\ar[d]^{\theta_{a,b}r_c}\\
	r_{a+b+c-d}\ar[rr]_{\phi_{a+b-d,c}}&& r_{a+b-d}r_c
	}
	\end{equation}
	where $d\le a+b$ and where moreover we allow the degenerate cases $b+c=d$ in which case we put
	$\theta_{a,d}=\Id_{r_a}$ and $a+b=d$ in which case we put $\phi_{0,c}=\Id_c$.
	\end{enumerate}
	
	\begin{propositions}
	\label{ref-3.3.1-80}
	Every morphism $f$ in $\Uscr$ can be written uniquely as a composition
	$
	f_{\u}\circ f_{\d}
	$
	with $f_{\d}$ in $\Uscr_{\d}$ and $f_{\u}$ in $\Uscr_{\u}$.
	\end{propositions}
	\begin{proof} It is clear that every morphism can be written as
	  $f_{\u}\circ f_{\d}$. To prove uniqueness we have to show that
	  rewriting a product of generators $p_{\d}\circ p_{\u}$ in the form
	  $p_{\u}'\circ p'_{\d}$ using the diagrams
	  \eqref{ref-3.14-78}\eqref{ref-3.15-79} and the monoidal properties of $\Uscr$ is compatible with
	the relations in $\Uscr_{\u}$ \eqref{ref-2.8-25} and in $\Uscr_{\d}$ \eqref{ref-3.13-77}. This is a tedious,
	but finite verification. Let us give one example. Consider
		\begin{equation}
		\label{ref-3.16-81}
		(ur_ar_b\theta_{c,e}v)\circ (u\phi_{a,b}r_cr_d^{-1}r_ev)\circ (u\phi_{a+b,c}r^{-1}_dr_ev):ur_{a+b+c}r_d^{-1}r_ev\r ur_ar_br_{c+e-d}v 
		\end{equation}	
	This can be rewritten in the form $f_{\u}\circ g_{\d}$ as follows.
		\begin{align}
		&[(ur_ar_b\theta_{c,e}v)\circ (u\phi_{a,b}r_cr_d^{-1}r_ev)]\circ (u\phi_{a+b,c}r^{-1}_dr_ev)\\
		&=(u\phi_{a,b} r_{c+e-d}v)\circ [(u r_{a+b} \theta_{c,e} v)\circ (u\phi_{a+b,c}r^{-1}_dr_ev)]	\notag\\
		&=(u\phi_{a,b} r_{c+e-d}v)\circ (u \phi_{a+b,c+e-d}           v) \circ (u\theta_{a+b+c,e}v) \label{ref-3.18-82}
		\end{align}
	However by the relations in $\Uscr_{\u}$ we have that \eqref{ref-3.16-81} is equal to
	\[
	(ur_ar_b\theta_{c,e}v)\circ
	(ur_a\phi_{b,c}r_d^{-1}r_ev)\circ (u\phi_{a,b+c}r_d^{-1}r_ev)
	\]
	which can be rewritten as
	\begin{align}
	&[(ur_ar_b\theta_{c,e}v)\circ (ur_a\phi_{b,c}r_d^{-1}r_ev)]\circ (u\phi_{a,b+c}r_d^{-1}r_ev) \\
	&=  (ur_a\phi_{b,c+e-d}v)\circ                           (ur_a\theta_{b+c,e}v)           \circ (u\phi_{a,b+c}r_d^{-1}r_ev)\notag\\
	&= (ur_a\phi_{b,c+e-d}v)\circ    (u\phi_{a,b+c+e-d}v)     \circ   (u\theta_{a+b+c,e}v)\label{ref-3.20-83}
	\end{align}
	and using the relations in $\Uscr_{\u}$ we see that \eqref{ref-3.18-82} is indeed equal to \eqref{ref-3.20-83}.
	\end{proof} 

\begin{remarks}
The category $\Uscr$ is an example of a Reedy category, see~\cite{riehl-verity}.
\end{remarks}

	\begin{lemmas} \label{ref-3.3.3-84}
$\Uscr$ is a rigid monoidal category.
	\end{lemmas}
	\begin{proof} It is enough to exhibit left and right duals for the generating objects
	$(r_a)_a$.
		\begin{enumerate}
		\item ${}^\ast r_d=r_d^\ast=r_d^{-1}$ and the unit/counit morphisms are the identity.
		\item Assume $a\neq d$. Then $r^\ast_a=r_{d-a}r_a^{-1}$. The unit $\eta:1\r r_a r_a^\ast=r_ar_{d-a} r_d^{-1}$ and counit $\epsilon: r_a^\ast\cdot r_a=r_{d-a}r_d^{-1}r_a\r 1$ 		are given by
		\begin{align*}
		\eta&=\phi_{a,d-a}r_d^{-1}\\
		\epsilon&=\theta_{d-a,a}
		\end{align*}
		\item Assume $a\neq d$. Then ${}^\ast r_a=r_a^{-1}r_{d-a}$. The unit $\eta:1\r {}^\ast r_a r_a=r_d^{-1}r_{d-a}r_{a}$ and counit $\epsilon: r_a \cdot {}^\ast r_a=r_{a}r_d^{-1}		r_{d-a}\r 1$ are given by
		\begin{align*}
		\eta&=r_d^{-1}\phi_{d-a,d}\\
		\epsilon&=\theta_{a,d-a}
		\end{align*}
		\end{enumerate}
	The fact that $\eta$, $\epsilon$ satisfy the required compatibilities follows from the relations \eqref{ref-3.14-78}\eqref{ref-3.15-79}.
	\end{proof}

\begin{remarks} 
\begin{enumerate}
\item One checks on generators (both objects and morphisms!) that as functors $\Uscr\r \Uscr^\circ$:
\[
(-)^\ast=r_d\cdot {}^\ast(-)r_d^{-1}
\]

\item The functors ${}^\ast(-)$, $(-)^\ast$ 
restrict to inverse isomorphisms $\Uscr_{\u}\leftrightarrow \Uscr_{\d}^\circ$. The functor $(-)^\ast:\Uscr_{\u}\r \Uscr_{\d}^\circ$
coincides with the identically named one introduced in the beginning of this section as part of the definition of~$\Uscr_{\d}$.
\end{enumerate}
\end{remarks}

\begin{remarks}
The category $\Uscr$ contains non-split projectors, ever after linearising. An example for the case $d=2$ is the composition
	\begin{equation}
	r_1r_1r_2^{-1}r_1 \xrightarrow{r_1 \theta_{11}} r_1 \xrightarrow{\phi_{11}r_2^{-1}r_1} r_1r_1r_2^{-1} r_1.
	\end{equation}
\end{remarks}

\section{From AS-regular algebras to fiber functors on $\Uscr$}
\label{ref-4-85}

Let $A=TV/(R)$ be a Koszul Artin-Schelter regular algebra of global
dimension $d$ (see Section~\ref{ref-2.7-32}) and let $M^+$ be the  monoidal functor
\[
M^+:\Uscr_{\u}^+ \r \Vect:r_a\mapsto R_a
\]
where $R_1=V$ and for $a\ge 2$:
\[
R_a=\bigcap_{i+j+2=a} V^{\otimes i} \otimes R\otimes V^{\otimes i}
\]
which sends $\phi_{a,b}$ to the inclusion $R_{a+b}\hookrightarrow R_aR_b$. This obviously respects the relations in $\Uscr_\u^+$ so $M^+$ really is a functor. Artin-Schelter regularity is crucial for the following result.

	\begin{proposition} 
	The functor $M^+$ can be uniquely extended to a  monoidal functor 
	\[
	M:\Uscr\r \Vect.
	\]
	\end{proposition}
	\begin{proof}
	This is a mildly tedious but straightforward verification. For any object $u \in \Uscr$, fix an expression of $u$ as a product of generators $(r_a)_a$ of $\Lambda$, and define $M(x)$ to be the tensor product of the corresponding $R_a$ and $R_d^{-1}$. The importance of AS-regularity comes from the need to define $M(\theta_{d-a,a})$. From the inclusion $R_d \to R_a R_{d-a}$ we get a morphism $R_{d-a}^* \to R_d^{-1}R_a$, and this has to be inverted to define $M(\theta_{d-a,a}):R_{d-a}R_d^{-1}R_a \to k$. AS-regularity ensures this can be done.
	\end{proof}
\begin{corollary} 
\label{ref-4.2-86} 
Let $f:u_1\r u$, $g:u_2\r u$ be morphisms in $U_{\u}\,\tilde{}$. Then 
\[
\im M(f\times_u g)=\im M(f)\cap \im M(g)\subset M(u)
\]
with the convention $M(\ast)=0$.
\end{corollary}
\begin{proof}
In view of Lemma \ref{ref-3.1.5-66}(4) it suffices to prove this for
$f,g\in \Uscr_{\u,1}$ where it is an easy verification which ultimately boils down to the fact that 
$
R_{a+b+c}=R_a R_{b+c}\cap R_{a+b}R_c
$
which holds for an arbitrary Koszul algebra.
\end{proof}
	\begin{proposition}
	\label{ref-4.3-87}
	The restricted functor $M:\Uscr_\u \to \Vect$ can be extended to an exact monoidal functor $\overline{M}:\mod((\Uscr_{\u}\,\tilde{}\,)^{\circ}) \to \Vect$ such that
$\overline{M}(\Uscr_{\u}\,\tilde{}\,(-,\ast))=0$.
	\end{proposition}
	\begin{proof}
We extend $M$ to a functor $\Uscr_{\u}\,\tilde{}\r \Vect$ by putting $M(\ast)=0$.
	Define $\overline{M}$ as the unique right exact functor $\overline{M}:
\mod((\Uscr_{\u}\,\tilde{}\,)^{\circ}) \to \Vect$, satisfying $k\Uscr_\u\,\tilde{}\,(-,u) \mapsto M(u)$. 
 We will now show that all the simples $S_u$ are $\overline{M}$-acyclic. Applying 
	$\overline{M}$ to the resolution~\eqref{ref-3.12-75} we get a complex
		\begin{multline}
\label{ref-4.1-88}
		\cdots \r\bigoplus_{i_1<\cdots<i_n} M({\operatorname{source} (f_{i_1}\times_u \dots \times_u f_{i_n})}) \r \cdots \\
		\cdots \r \bigoplus_{i<j} M({\operatorname{source}(f_i\times_x f_j)})\r \bigoplus_{f_i:u_i\r u\in \Uscr_{\u,1}}M(u_i)\r M(u) \r 0,
		\end{multline}
Since by Corollary \ref{ref-4.2-86} we have
\[
		 \im M(f_{i_1}\times_u \dots \times_u f_{i_n})=\bigcap_j \im M(f_{i_j}),
\]
exactness of \eqref{ref-4.1-88} follows from Proposition \ref{ref-2.3.2-14} if 
the subspaces  
		\begin{equation}
		\{\im M(f_i) | \exists f_i:u_i \to u \text{ in } \Uscr_{\u,1}\} \subset M(u)
		\end{equation}
                generate a distributive lattice.  To prove this we may by Remark~\ref{ref-3.1.10-76} assume that
                $u\in\Uscr^+_{\u}$. But then it follows from the Koszul property
of $A$.
	\end{proof}
	\begin{proposition}
	\label{ref-4.4-89}
Let $u\in \Ob(\Uscr)$. The following collections of subspaces 
		\begin{equation}
\label{ref-4.3-90}
\begin{aligned}
\im M(f)&\subset  M(u)\\
\ker M(g)&\subset  M(u)
\end{aligned}
		\end{equation} for $f:v\r u$  and $g:u\r w$ in $\Uscr$
both generate distributive lattices in $M(u)$.
	\end{proposition}
	\begin{proof}
The second claim follows from the first one by duality ($M$ being a monoidal functor is compatible with duals)
and Proposition \ref{ref-2.3.1-13}(2).

We now prove the first claim.	By Proposition~\ref{ref-3.3.1-80}, $f=f_{\u} \circ f_{\d}$ for $f_{\u} \in \Uscr_{\u}$ and $f_{\d} \in \Uscr_{\d}$. Since the $M(f_{\d})$ are
	all surjective, we may restrict to the (finite!) set $\{f:v\r u\in \Uscr_{\u}\}:=\{f_1,\ldots,f_u\}$. 
Applying the exact functor $\overline{M}$ introduced in Proposition \ref{ref-4.3-87} to \eqref{ref-3.11-74} and using Corollary
\ref{ref-4.2-86}
we get an exact sequence
		\begin{multline}
		\cdots \r\bigoplus_{i_1<\cdots<i_n} \bigcap_{j} \im M(f_{i_j})
\r \cdots \\
		\cdots \r \bigoplus_{i<j} \im M(f_i)\cap \im M(f_j)\r \bigoplus_{f_i}\im  M(f_i)\r M(u) \r 0,
		\end{multline}
It now suffices to apply
	Proposition~\ref{ref-2.3.2-14}.
	\end{proof}

\section{Recovering $\uaut(A)$}
\label{ref-5-91}

Here we will show that it is possible to reconstruct $\uend(A)$ and $\uaut(A)$ from the categories $\Uscr_{\u}^+$, $\Uscr_{\u}$ that were defined in the previous sections.

	\begin{theorem}
	\label{ref-5.1-92}
	Both $\uend(A)$ and $\uaut(A)$ can be reconstructed as follows
		\begin{align}
		\label{ref-5.1-93}
		\uend(A) & \cong \coend_{\Uscr_{\u}^+}(M^+), \\
		\label{ref-5.2-94}
		\uaut(A) & \cong \coend_{\Uscr}(M).
		\end{align}
	These are both isomorphisms of bialgebras (and thus in the second case also of Hopf algebras).
	\end{theorem}

To prove this theorem, we use the Tannakian setup. Let $\Cscr_1$ denote the free strict monoidal category with presentation
	\begin{equation}
	\Cscr_1=\langle r_1,r_2 \ \vert \ r_2 \to r_1r_1 \rangle,
	\end{equation}
and let $M^+:\Cscr_1 \to \Vect$ denote the monoidal functor defined by 
	\begin{align}
	r_1 & \mapsto V, \\
	r_2 & \mapsto R, \\
	(r_2 \to r_1r_1) & \mapsto (R \hookrightarrow VV).
	\end{align}
With this data, we can construct a bialgebra $\coend_{\Cscr_1}(M^+)$ and the starting point for the proof of Theorem~\ref{ref-5.1-92} will be the following description of $\uend(A)$ which follows from Definition~\ref{ref-2.8.1-35}.
	
	\begin{proposition}
          \label{ref-5.2-95}
	There is an isomorphism of bialgebras 
		\begin{equation}
		\uend(A) \cong \coend_{\Cscr_1}(M^+).
		\end{equation}
	\end{proposition}
		
	\begin{proof}[Proof of Theorem~\ref{ref-5.1-92}]
	We first prove~\eqref{ref-5.1-93}. 
Note that we have the following commutative diagram
		\begin{equation}
		\begin{tikzcd}
		\Cscr_1 \rar[hook] \dar[swap]{M^+} & \Uscr_\u^+ \dlar{M^+} \\
		\Vect
		\end{tikzcd}
		\end{equation}
           Passing to duals we see that by Proposition \ref{ref-5.2-95} it is sufficient to prove that the restriction morphism
\begin{equation}
\label{ref-5.9-96}
\End_{\Uscr_\u^+}(M^+) \r \End_{\Cscr_1}(M^+)
\end{equation}
is an isomorphism.
Put
\[
  \Cscr_2=\langle r_1, \ldots, r_d \ \vert \phi_{a,b}:r_{a+b} \to r_ar_b \rangle
\]
In other words $\Cscr_2$ is obtained from $\Uscr_{\u}^+$ by forgetting all relations. Obviously
\begin{equation}
\label{ref-5.10-97}
\End_{\Cscr_2}(M^+) = \End_{\Uscr_\u^+}(M^+) 
\end{equation}
where $M^+$ on the left is the composition $\Cscr_2\r U_{\u}^+\xrightarrow{M^+}\Vect$.

Combining \eqref{ref-5.9-96} and \eqref{ref-5.10-97} we see that we have to show that 
the composed restriction morphism
		\begin{equation}
		\End_{\Cscr_2}(M^+) \xrightarrow{\res} \End_{\Cscr_1}(M^+).
		\end{equation}
is an isomorphism.
	In fact this map has an inverse: a natural transformation~$\alpha \in \End_{\Cscr_1}(M^+)$ corresponds to giving for each $n$
an	endomorphism~$\alpha_{V^n}:V^n \to V^n$ such that~$\alpha_{V^n}(V^iRV^j) \subset V^iRV^j$. As $M^+(r_i)=\cap_{p+q+2=i} V^p R V^q$
such an $\alpha$ also satisfies~$\alpha_{V^n}(M^+(r_i)) \subset 
	M^+(r_i)$, so we can uniquely extend~$\alpha$ to a 	natural transformation in~$\End_{\Cscr_2}(M^+)$. This defines the inverse of $\res$.
	
	We now prove~\eqref{ref-5.2-94}. 
        Let $\Uscr^{+,\dagger}_{\u}$ be obtained from $\Uscr^{+,\ast}_{\u}$
        by inverting the unit and counit morphism associated with
        $r_d^\ast$. Then the embedding $\Uscr^+_{\u}\r \Uscr_{\u}$ lifts to a functor $\Uscr^{+,\dagger}_{\u}\r \Uscr_{\u}^\ast$ which is easily seen to be an equivalence. The functor $M^+$ extends uniquely to a functor $M^{+,\dagger}:\Uscr^{+,\dagger}_{\u}\r \Vect$ 
and we obtain isomorphisms
\begin{equation}
\label{ref-5.12-98}
\End_{\Uscr_{\u}^\ast}(M)\cong \End_{\Uscr_{\u}^{+,\dagger}}(M^{+,\dagger})\cong \End_{\Uscr_{\u}^{+,\ast}}(M^{+,\ast})
\end{equation}
For the lefthand side of \eqref{ref-5.2-94} we note that by 
Corollary~\ref{ref-2.8.5-38}, \eqref{ref-5.1-93},  \eqref{ref-5.12-98} and Proposition~\ref{ref-2.9.7-44} we have
\begin{equation}
\label{ref-5.13-99}
\uaut(A)=H(\uend(A))=H(\coend_{\Uscr_\u^+}(M^+))=\coend_{\Uscr^{+,\ast}_\u}(M^{+,\ast})=\coend_{\Uscr_\u^*}(M^{*})
\end{equation}
By Lemma~\ref{ref-2.9.6-43} there is a monoidal functor $I^*$ making the two triangles in the diagram
		\begin{equation}
		\begin{tikzcd}
		\Uscr_{\u} \rar{*} \dar[swap]{I} & \Uscr_{\u}^* \dlar[dashed,swap]{I^*} \dar{M^*} \\
		\Uscr \rar{M} & \Vect
		\end{tikzcd}
		\end{equation}
	commute, where $I$ is the embedding. The map $I^\ast$ then yields a restriction morphism
\begin{equation}
\label{ref-5.15-100}
 \End_{\Uscr}(M)\r \End_{\Uscr_\u^*}(M^{*})
\end{equation}
We will show below it is an isomorphism. Dualising this isomorphism and combining it with \eqref{ref-5.13-99} 
we
get a composed isomorphism
\[
\uaut(A)\xrightarrow{\eqref{ref-5.13-99}} \coend_{\Uscr_\u^*}(M^{*})\xrightarrow{\eqref{ref-5.15-100}} \coend_{\Uscr}(M)
\]
which finishes the proof of \eqref{ref-5.2-94}.

We now prove that \eqref{ref-5.15-100} is indeed an isomorphism. 
Define $\pi_a:r^\ast_{d-a}\r r_ar_d^{-1}$ in $\Uscr^\ast_{\u}$ as the composition
\[
r^\ast_{d-a} =r^\ast_{d-a}r_d r_{d}^{-1}\xrightarrow{r^\ast_{d-a} \phi_{d-a,a} r^{-1}_d} r^\ast_{d-a} r_{d-a} r_ar_d^{-1}\xrightarrow{\eval_{d-a}r_ar_d^{-1}} r_ar_d^{-1}
\]
where $\eval_{d-a}$ is the counit morphism. 

Let $\Uscr^{\dagger}_\u$ be obtained  from $\Uscr^{\ast}_\u$ by formally inverting the $\pi_a$ for all $a$. The functor  $I^\ast: \Uscr^\ast_\u\r \Uscr$ inverts $(\pi_a)_a$
and so it extends to a monoidal functor
$I^\dagger: \Uscr^\dagger_\u\r \Uscr$. We claim this functor is full. We can
clearly lift the generators $\phi_{a,b}$ of $\Uscr$ under $I^\dagger$ and to lift the generators $\theta_{a,b}$ we write them in the form $\theta_{a,d-a}r_{a+b-d}\circ r_ar_d^{-1}\phi_{d-a,a+b-d}$. Hence we only have to lift $\theta_{a,d-a}$ and such a lift is given by the following composition
\[
r_ar_d^{-1}r_{d-a}\xrightarrow{\pi_a^{-1}}  r^\ast_{d-a}r_{d-a}\xrightarrow{\eval_{d-a}} 1
\]
From 
AS-regularity
of $A$ one also obtains that $M^*(\pi_a)$ is invertible and hence $M^*$ extends uniquely to a monoidal functor
$M^\dagger: \Uscr^\dagger_{\u}\r \Vect$.
So we now have morphisms of monoidal categories
\[
\Uscr^\ast_{\u}\r \Uscr^\dagger_\u\xrightarrow{I^\dagger} \Uscr
\]
which yields restriction morphisms
\[
\End_{\Uscr}(M)\xrightarrow{\cong} \End_{\Uscr^\dagger_\u}(M^\dagger)\xrightarrow{\cong} \End_{\Uscr^\ast_{\u}}(M^\ast)
\]
whose composition is \eqref{ref-5.15-100}. The second arrow is an isomorphism because
of the fullness of $I^\dagger$.
	\end{proof}
\begin{corollary} $\uaut(A)$ has invertible antipode.
\end{corollary}
\begin{proof} This follows from Theorem \ref{ref-5.1-92}, Lemma \ref{ref-3.3.3-84} and Proposition \ref{ref-2.9.4-42}.
\end{proof}

\section{$\uaut(A)$ is quasi-hereditary}
\label{ref-6-101}

In this section we show that~$\uaut(A)$ is quasi-hereditary as in Definition~\ref{ref-2.11.2-58}, by using the Tannaka-Krein formalism. In particular $(\Uscr,M)$ denotes the pair defined in Sections~\ref{ref-3-62} and~\ref{ref-4-85}. Remember that $\Lambda$ denotes the monoid $\langle r_1,\ldots,r_{d-1},r_d^{\pm 1} \rangle$ and $\Lambda=\Ob(\Uscr)$. Also for $w \in \Lambda$, $l(w)$ is the length of $w$ written in the shortest possible way as a word in $r_1, \ldots, r_d,r_d^{-1}$.
We equip $\Lambda$ with the minimal left and right invariant ordering satisfying	
				\begin{equation}
				r_{a+b}<r_ar_b, \qquad 1<r_ar^{-1}_dr_{d-a}
				\end{equation}
			for $a,b\ge 1$, $a+b\le d$.

\subsection{Standard and costandard comodules}

Let us define  candidate standard, costandard and simple comodules. By Theorem~\ref{ref-5.1-92}, there is an evaluation functor
	\begin{equation}
	\eval_M:\Uscr \to \comod(\uaut(A)),
	\end{equation}
	so for any~$\lambda \in \Lambda$, we get an $\uaut(A)$-comodule $M(\lambda)$. We define
	\begin{equation}
	\label{ref-6.3-102}
	\begin{aligned}
	\nabla(\lambda) &= \coker \bigg( \bigoplus_{\substack{\mu \to \lambda \in \Uscr\\ \mu < \lambda}} M(\mu) \to M(\lambda) \bigg)
	\end{aligned}
	\end{equation}
and similarly
	\begin{equation}
	\label{ref-6.4-103}
	\begin{aligned}
	\Delta(\lambda) &= \ker \bigg( M(\lambda) \to \bigoplus_{\substack{\lambda \to \mu \in\Uscr\\ \mu < \lambda}} M(\mu) \bigg).
	\end{aligned}
	\end{equation}
	and moreover we let $L(\lambda)$ be the image of the composition
\begin{equation}
\label{ref-6.5-104}
\Delta(\lambda)\r M(\lambda)\r \nabla(\lambda)
\end{equation}
	\begin{propositions}
	\label{ref-6.1.1-105}
	With the above definitions, one has
	\begin{equation}
	\begin{aligned}
	M(\lambda^*) &\cong M(\lambda)^*, & \nabla(\lambda)^* &\cong \Delta(\lambda^*),& L(\lambda)^\ast\cong L(\lambda^\ast) \\
	M({}^*\lambda) &\cong {}^*M(\lambda), & {}^*\nabla(\lambda) &\cong \Delta({}^*\lambda)&{}^\ast L(\lambda)\cong L({}^\ast\lambda) 
	\end{aligned}
	\end{equation}
	\end{propositions}
	\begin{proof}
	By definition of the~$M(\lambda)$, it suffices to note that $M$ is monoidal and monoidal functors 
	preserve duals. 
	
	For the other part, right dualising the exact sequence defining~$\nabla(\lambda)$ gives an exact sequence:
		\begin{equation}
		0 \to \nabla(\lambda)^* \to M(\lambda)^* \to \bigoplus_{\mu \to \lambda} M(\mu)^*.
		\end{equation}
	Now using the first part and the fact that in a rigid monoidal category, taking right duals is fully faithful we get the
	exact sequence
	 	\begin{equation}
		0 \to \nabla(\lambda)^* \to M(\lambda^*) \to \bigoplus_{\lambda^* \to \mu^*} M(\mu^*),
		\end{equation}
	which is exactly the definition of $\Delta(\lambda^*)$. For left duals the proof is completely similar.
	\end{proof}

	\begin{lemmas}
	\label{ref-6.1.2-106}
	To compute $\nabla(\lambda)$ and $\Delta(\lambda)$, it suffices to sum over all morphisms in $\Uscr_{\u,1}$, respectively $\Uscr_{\d,1}$.
	\end{lemmas}
	\begin{proof}
	Let's consider $\nabla(\lambda)$, the other case is analogous. By Proposition~\ref{ref-3.3.1-80} any morphism $f:\mu \to \lambda$ in $\Uscr$ has a factorisation 
	$f=f_\u \circ f_\d$ for $f_\u$ a morphism in $\Uscr_\u$ and $f_\d$ a morphism in $\Uscr_\d$. Since $M(f_{\d})$ is surjective, the 
	statement is clear.
	\end{proof}	
	\begin{corollarys}
	\label{ref-6.1.3-107}
	Let $\lambda=r_d^{a_0}\lambda_1 r_d^{a_1} \lambda_2 r_d^{a_2} \cdots r_d^{a_{k-1}}\lambda_k r_d^{a_k}$, where $\lambda_i \in \langle r_1, \ldots, r_{d-1} \rangle$ and $a_i \neq 0$, then
		\begin{equation}
		\nabla(\lambda)=\nabla(r_d)^{a_0}\nabla(\lambda_1)\nabla(r_d)^{a_1}\nabla(\lambda_2)\nabla(r_d)^{a_2} \cdots 
		\nabla(\lambda_k)\nabla(r_d)^{a_k}.
		\end{equation}
	with in particular, $\nabla(r_d)=M(r_d)=R_d$.
	\end{corollarys}
\begin{proof} This follows from Lemma \ref{ref-6.1.2-106} combined with Remark \ref{ref-3.1.10-76}.
\end{proof}
	
	\begin{remarks}
	Note that this corollary says that at least on the level of the $\nabla$'s, the representations of $\uaut(A)$ behave as $\Oscr(\GL_n)$, where one can factor out the determinant representation. 
	\end{remarks}
We will consider a refinement of Corollary \ref{ref-6.1.3-107}
Assume that $u$ and $v\in \Ob(\Uscr)$ when written in shortest possible form are given by $wr_ar_d^{-x}$, $r_d^xr_bw'$ and $a+b\le d$, $x\in \ZZ$. Then we put
\[
u\wedge v:=wr_{a+b}w'
\] 
In all other case we consider $u\wedge v$ to be undefined. It will be convenient to put $u\wedge v=\ast$ in that case. 

There is a canonical map in $\Uscr$
\[
\phi_{u,v}:u\wedge v\r uv
\]
derived from $\phi_{a,b}$. 

Similarly assume that $u$ and $v\in \Ob(\Uscr)$ when written in shortest possible form are given by $wr_ar_d^{-x-1}$, $r_d^xr_bw'$ and $a+b\ge d$, $x\in \ZZ$. Then we put
\[
u\vee v:=wr_{a+b-d}w'
\]
(with the convention $r_0=1$). In all other case we consider $u\wedge v$ to be undefined. As above we put $u\wedge v=\ast$ in that case.
There is a canonical map in $\Uscr$
\[
\theta_{u,v}:uv\r u\vee v
\]
derived from $\theta_{a,b}$. The operations $\wedge$, $\vee$ are related by duality
\[
(u\wedge v)^\ast=v^\ast\vee u^\ast
\]
and similarly $\phi_{u,v}^\ast=\theta_{v^\ast,u^\ast}$.
\begin{propositions} \label{ref-6.1.5-108} For $u,v\in \Ob(\Uscr)$
there is a commutative diagram in $\comod(\uaut(A))$ with an exact lower row 
\begin{equation}
\label{ref-6.10-109}
\xymatrix{
&M(u\wedge v)\ar@/^2em/[rr]^{M(\phi_{u,v})}\ar@{->>}[d]\ar@{^(->}[r] &M(u)\otimes M(v)\ar@{->>}[d]\ar@{=}[r]& M(uv)\ar@{->>}[d]\\
0\ar[r]&\nabla(u\wedge v)\ar[r] & \nabla(u)\otimes \nabla(v)\ar[r]& \nabla(uv)\ar[r]&0
}
\end{equation}
with the convention that $M(\ast)=\nabla(\ast)=0$. There is an analogous dual
commutative diagram given by
\begin{equation}
\label{ref-6.11-110}
\xymatrix{
0\ar[r]&\Delta(uv)\ar@{^(->}[d]\ar[r] & \Delta(u)\otimes \Delta(v)\ar@{^(->}[d]\ar[r]\ar[r]& \Delta(u\vee v)\ar@{^(->}[d]\ar[r]\ar[r]&0\\
&M(uv)\ar@/_2em/[rr]_{M(\theta_{u,v})}\ar@{=}[r] &M(u)\otimes M(v)\ar@{->>}[r]& M(u\vee v)
}
\end{equation}
where again $M(\ast)=\Delta(\ast)=0$ and now the upper row is exact.
\end{propositions}
\begin{proof} 
We will only consider \eqref{ref-6.10-109}. The existence of \eqref{ref-6.11-110} then follows by duality.
All the objects in \eqref{ref-6.10-109} as well as the morphisms in the top row and the vertical morphisms are in $\comod(\uaut(A))$. 
To prove the existence of the horizontal morphisms in the lower row \emph{as linear maps} we may
use Corollary \ref{ref-6.1.3-107}  to reduce to the
case $u,v\in \Uscr_{\u}^+$ in which case it follows from \eqref{ref-2.14-31}. Since the vertical maps are surjective we obtain that the lower horizontal maps
must be compatible with the $\uaut(A)$-structure. So the diagram lives in $\comod(\uaut(A))$.
\end{proof}
Denote by $\Fscr(\Delta)$ (respectively $\Fscr(\nabla)$) the categories of $\uaut(A)$-comodules that have a $\Delta$-filtration (respectively $\nabla$-filtration).
	\begin{corollarys}
	\label{ref-6.1.6-111}
\begin{enumerate}
\item \label{ref-1-112} $\Fscr(\Delta)$ and $\Fscr(\nabla)$ are closed under tensor products. 
\item \label{ref-2-113} $M(\lambda)\in \Fscr(\Delta)\cap \Fscr(\nabla)$.
\end{enumerate}
\end{corollarys}
	\begin{proof}
\begin{enumerate}
\item
This is immediate from
Proposition \ref{ref-6.1.5-108}. 
\item
It follows from (1) that $\Fscr(\Delta)\cap \Fscr(\nabla)$ is also closed under tensor product.
Since $M(r_i)= \nabla(r_i)=\Delta(r_i)\in \Fscr(\nabla)\cap \Fscr(\Delta)$ it follows that every $M(\lambda)$
is in $\Fscr(\nabla)\cap \Fscr(\Delta)$. \qed
\end{enumerate}
\def\qed{}	\end{proof}
	\begin{propositions}
\label{ref-6.1.7-114}
	For any~$\lambda \in \Lambda$, the image of the composition of the canonical inclusion and surjection
		\begin{equation}
		\label{ref-6.12-115}
		\Delta(\lambda) \hookrightarrow M(\lambda) \twoheadrightarrow \nabla(\lambda)
		\end{equation}
	is non-zero. In particular,~$\Delta(\lambda) \neq 0$, $\nabla(\lambda)\neq 0$, $L(\lambda)\neq 0$.
	\end{propositions}
	\begin{proof}
	By Lemma~\ref{ref-6.1.1-105} we already know that~$\Delta(\lambda^*) \cong \nabla(\lambda)^*$, so also 
	$\nabla({}^*\lambda) \cong {}^*\Delta(\lambda)$ and by tensoring~\eqref{ref-6.12-115} with ${}^*\Delta(\lambda)$ on 
	the left and precomposing with $\coeval_{\Delta(\lambda)}$, it suffices to check that the morphism
		\begin{equation}
		\label{ref-6.13-116}
		k \xrightarrow{i} \nabla({}^*\lambda) \ot \nabla(\lambda)
		\end{equation}
	induced by~\eqref{ref-6.12-115} is non-zero. The morphism $i$ is characterized by the fact that it fits in commutative
diagram of the form:
		\begin{equation}
		\label{ref-6.14-117}
		\begin{tikzcd}
		k \rar[equals] \dar[swap]{M(\coeval_{\lambda})} & k \dar{i} \\ M({}^*\lambda)M(\lambda) \rar & \nabla({}^*\lambda)\nabla(\lambda)
		\end{tikzcd}
		\end{equation}
	
	Consider the following possibilities for $\lambda$:
	\begin{enumerate}
		\item $\lambda$ starts with a strictly negative power of $r_d$. Put $\lambda=r_d^{-1}\lambda'$: 
			in this case, ${}^*\lambda={}^*\lambda' r_d$ and using 
			Corollary~\ref{ref-6.1.3-107}, we get
				\begin{align}
				\nabla({}^*\lambda')\nabla(\lambda') &= \nabla({}^*\lambda'r_d)\nabla(r_d^{-1})\nabla(\lambda') \\
				&= \nabla({}^*\lambda)\nabla(\lambda).
				\end{align}
                              \item $\lambda$ starts with a strictly positive power of $r_d$. This case is similar.
		\item $\lambda=r_i\lambda'$, $1\le i\le d-1$ and $\lambda'$ does not start with a strictly negative power of $r_d$:
			in this case, ${}^*\lambda={}^*\lambda' r_d^{-1} r_{d-i}$.
				\begin{align}
				\nabla({}^*\lambda')\nabla(\lambda') &= \nabla({}^*\lambda')\nabla(r_d^{-1}r_d\lambda') \\
				&= \nabla({}^*\lambda')\nabla(r_d^{-1})\nabla(r_d\lambda') 
				\end{align}
			By Proposition~\ref{ref-6.1.5-108}, this embeds into
				$$
				\nabla({}^*\lambda')\nabla(r_d^{-1})\nabla(r_{d-i})\nabla(r_i\lambda')=\nabla({}^*\lambda'r_d^{-1}r_{d-i})\nabla(r_i\lambda')=
\nabla({}^*\lambda)\nabla(\lambda).
				$$
 where in the first equality we have used Corollary \ref{ref-6.1.3-107} and the easily verified fact that in this case ${}^\ast\lambda'$ does not end with a strictly positive power of $r_d$.
		\item $\lambda=r_i\lambda'$ and $\lambda'$ does start with a strictly negative power of $r_d$. Put 
			 $\lambda'=r_d^{-1}\lambda''$, so ${}^*\lambda={}^*\lambda'' r_{d-i}$ 
				\begin{equation}
				\nabla({}^*\lambda')\nabla(\lambda') = \nabla({}^*\lambda'' r_d) \nabla(\lambda')
				\end{equation}
			which by Proposition~\ref{ref-6.1.5-108} is contained in 
				\begin{equation}
				\nabla({}^*\lambda'' r_{d-i})\nabla(r_i)\nabla(\lambda') = \nabla({}^*\lambda'' r_{d-i})\nabla(r_i\lambda')
=\nabla({}^*\lambda)\nabla(\lambda)
				\end{equation}
		where we also use Corollary \ref{ref-6.1.3-107} for the first equality.
	\end{enumerate}
	
	Summarising, for all $\lambda \in \Lambda$ we have constructed an injective map
		\begin{equation}
		\nabla({}^*\lambda')\nabla(\lambda') \hookrightarrow \nabla({}^*\lambda)\nabla(\lambda)
		\end{equation}
	where $\lambda'$ has length $l(\lambda)-1$. Repeating this process and composing we eventually end up with $\lambda'=1$. Since $\nabla(1)=k$ we have obtained an
injective map  $k\r  \nabla({}^*\lambda)\nabla(\lambda)$. We still have to verify that it is the same map as in \eqref{ref-6.13-116}.

To see this we note that
	the exact same process can be applied to the $M({}^*\lambda)M(\lambda)$. This allow us to write the diagram~\eqref{ref-6.14-117} as a 
	composition of commutative diagrams of the form
		\begin{equation}
		\label{ref-6.22-118}
		\begin{tikzcd}
		M({}^*\mu_i)M(\mu_i) \rar \dar & \nabla({}^*\mu_i)\nabla(\mu_i) \dar \\ M({}^*\mu_{i+1})M(\mu_{i+1}) \rar & \nabla({}^*\mu_{i+1})\nabla(\mu_{i+1})
		\end{tikzcd}
		\end{equation}
	for $0\leq l(\mu_i)<l(\mu_{i+1}) \leq l(\lambda)$. It remains to check that the composition of all the lefthand vertical maps in these stacked diagrams is the map 
	$k \to M({}^*\lambda)M(\lambda)$ induced by the coevaluation $1 \to {}^*\lambda \lambda$. This is an easy verification.
	\end{proof}
	\begin{examples}
AS-regularity is crucial for Proposition \ref{ref-6.1.7-114}.
	The algebra $A=k[x]/(x^2)$ is Koszul but $\nabla(r_1^2)=0$. This is because $A$ had infinite global dimension, so it is not AS-regular.
	\end{examples}

\subsection{The fundamental exact sequences}
In this section we fix the following setting: $\Lambda_1\subset \Lambda_2$ are two saturated subposets of $(\Lambda,<)$ such that the elements of $\Lambda_2-\Lambda_1$
are incomparable. Let $\Uscr_i$ be
the full subcategory of $\Uscr$ whose object set is~$\Lambda_i$.

\begin{remarks}
\label{ref-6.2.1-119} It follows from Lemma~\ref{ref-3.3.1-80}.
If  $f:u\r v$ is a morphism in $\Uscr$ with $u,v\in \Lambda_{1}$ then we may write $f$ as a product of
elements of $U_{\u,1}$ and $U_{\d,1}$ which are all contained in $\Uscr_{1}$.
\end{remarks}
\begin{remarks} If we compute 
$\Delta(\lambda),\nabla(\lambda),L(\lambda)$ with the formulas \eqref{ref-6.4-103}\eqref{ref-6.3-102}\eqref{ref-6.5-104}  for $\comod(\coend_{\Uscr_i}(M))$ 
we get the same result  as in $\comod(\coend_{\Uscr}(M))$, assuming  of course $\lambda\in \Lambda_i$. This will be used without further comment below.
\end{remarks}

	\begin{theorems}
	\label{ref-6.2.3-120}
	There is an exact sequence
	\begin{equation}
	\label{ref-6.23-121}
	0 \to \prod_{\lambda\in\Lambda_2-\Lambda_1}\Hom_k(\nabla(\lambda),\Delta(\lambda)) \xrightarrow{\alpha} \End_{\Uscr_2}(M) \xrightarrow{\beta} \End_{\Uscr_1}(M) \to 0,
	\end{equation}
where $\beta$ is the restriction morphism and	where for $f \in \Hom_k(\nabla(\lambda),\Delta(\lambda))$, $\alpha(f):M \r M$ is determined by
	\begin{enumerate}
		\item $\alpha(f)(\mu)=0$ if $\lambda \neq \mu$,
		\item $\alpha(f)(\lambda)=M(\lambda) \to \nabla(\lambda) \xrightarrow{f} \Delta(\lambda) \to M(\lambda)$.
	\end{enumerate}

	\end{theorems}
	\begin{proof}
	That $\alpha$ is a natural transformation follows from the definitions~\eqref{ref-6.3-102} and~\eqref{ref-6.4-103} of 
	$\nabla$ and~$\Delta$: indeed, for a map $\lambda \xrightarrow{p} \mu$ with $\mu \neq \lambda$ we should get a commuting diagram
		\begin{equation}
		\begin{tikzcd}
		M(\lambda) \dar{M(p)} \rar & \nabla(\lambda) \rar{f} & \Delta(\lambda) \rar & M(\lambda) \dar{M(p)} \\
		M(\mu) \arrow{rrr}{0} &  & & M(\mu)
		\end{tikzcd}
		\end{equation}
	but since by definition $\Delta(\lambda) \subset \ker(M(\lambda) \xrightarrow{M(p)} M(\mu))$, this is clear. The case $\mu \to \lambda$ is analogous and $\mu \to \mu'$, with $\mu,\mu' \neq \lambda$ is trivial. 
	Remains to check what happens for $\lambda \xrightarrow{p} \lambda$. By Proposition~\ref{ref-3.3.1-80}, $p$ can be factored as $p_\u \circ p_\d$ for $p_\u$ in $\Uscr_\u$ and $p_\d$ in 
	$\Uscr_\d$, and naturality now follows from the previous cases.

We now prove the surjectivity of $\beta$.
For a given $(\phi_u)_{u\in\Lambda_1}\in \End_{\Uscr_1}(M)$, it obviously suffices to be able to
          define all the $\phi_v$, for $v\in \Lambda_2-\Lambda_1$, in such a way as
          to obtain a natural transformation $M\r M$ over $\Uscr_2$

The defining property of a natural transformation only needs to be checked on generators. 
Taking into account that $(\phi_u)_{u\in \Lambda_1}$ already forms a natural transformation
it follows from Remark \ref{ref-6.2.1-119} below and the incomparability of $\Lambda_2-\Lambda_1$ that there are the following constraints on $(\phi_v)_{v\in \Lambda_2-\Lambda_1}$. 
\begin{enumerate} 
\item
For every $\alpha:u\r v\in \Uscr_{\u,1}$ with $v\in \Lambda_2$ (and thus necessarily $u\in \Lambda_1$) we have
a commutative diagram
\[
\xymatrix{
M(u)\ar[d]_{\phi_u}\ar[r]^{M(\alpha)} & M(v)\ar[d]^{\phi_v}\\
M(u)\ar[r]_{M(\alpha)} & M(v)
}
\]
\item
For every $\beta:v\r w\in \Uscr_{\d,1}$ with $v\in \Lambda_2-\Lambda_1$ (and thus necessarily $u\in \Lambda_1$) we have
a commutative diagram
\[
\xymatrix{
M(v)\ar[d]_{\phi_v}\ar[r]^{M(\beta)} & M(w)\ar[d]^{\phi_w}\\
M(v)\ar[r]_{M(\beta)} & M(w)
}
\]
\end{enumerate}
In other words finding each individual  $\phi_v$, for $v\in \Lambda_2-\Lambda_1$ (they do not interfere) is equivalent to finding $h$ in a diagram like \eqref{ref-6.25-123} below
where the $U_i$ are of the form $\im M(\alpha_i)$
for $\alpha_i:u_i\r v\in \Uscr_{\u,1}$ and the $W_j$ are of the form $\ker M(\beta_j)$ for $\beta_j:v\r w_j\in \Uscr_{\d,1}$.
Thus we must check the conditions for Lemma \ref{ref-6.2.4-122} below in this situation. We do this next.
\begin{enumerate}
\item Distributivity of $\im M(\alpha_i)$ for $\alpha_i:u_i\r v\in \Uscr_{\u,1}$. This follows from Proposition \ref{ref-4.4-89}.
\item Distributivity of $\ker M(\beta_j)$ for $\beta_j:v\r w_j\in \Uscr_{\d,1}$. This also follows from Proposition \ref{ref-4.4-89}.
\item Existence of commutative diagrams
\[
\xymatrix{
M(u_i)\ar[d]_{\phi_{u_i}} \ar[r]^{M(\alpha_i)}& M(v) \ar[r]^{M(\beta_j)} & M(w_j)\ar[d]^{\phi_{w_j}}\\
M(u_i) \ar[r]_{M(\alpha_i)}&M(v)\ar[r]_{M(\beta_j)} &M(w_j)
}
\]
This follows from the fact that $\beta_j\alpha_i\in \Uscr_{1}$ and the fact that the $(\phi_u)_{u\in\Lambda_1}$ form a natural transformation.
\item The $\phi_{u_i}$ agree on pairwise intersections. This follows from Corollary \ref{ref-4.2-86}.
\item The $\phi_{w_i}$ agree on pairwise pushouts. This follows from Corollary \ref{ref-4.2-86} by duality.
\end{enumerate}
So  the conditions for Lemma \ref{ref-6.2.4-122} are indeed satisfied which yields that  $\beta$ is indeed surjective.
	
To compute the kernel of $\beta$ we now assume that $(\phi_u)_{u\in \Lambda_1}=0$. By the above discussion and Lemma \ref{ref-6.2.4-122} we obtain that now $\phi_v$ for $v\in \Lambda_2-\Lambda_1$ is given as a composition
\[
M(v)\r \coker\left(\bigoplus_{u_i\r v\in \Uscr_{\u,1}} M(u_i)\r M(v)\right)\xrightarrow{\phi'_v} \ker\left(M(v)\r \bigoplus_{v\r w_i\in \Uscr_{\d,1}} M(w_j)\right)\r M(v)
\]
In other words $\phi'_v\in \Hom(\Delta(v),\nabla(v)$ and $(\phi_v)_{v\in\Lambda_2}=\alpha((\phi'_v)_v)$.
\end{proof}
We have  used the following linear algebra lemma. 
	\begin{lemmas}
	\label{ref-6.2.4-122}
        \begin{enumerate}
\item
	For a finite dimensional vector space $V$ with  distributive collections of subspaces $\{U_1, \ldots, U_k\}$, $\{W_1, \ldots, W_l\}$, consider the diagram
	\begin{equation}
	\label{ref-6.25-123}
		\begin{tikzpicture}[baseline=0pt]
		\node (0) {$\vdots$};
		\node (1) [above of = 0] {$U_1$};
		\node (2) [below of =0] {$U_k$};
		\node (3) [right of =0] {$V$};
		\node (4) [right of =3] {$\vdots$};
		\node (5) [above of =4] {$V/W_1$};
		\node (6) [below of =4] {$V/W_l$};
		
		\draw[->] (1) to (3);
		\draw[->] (2) to (3);
		\draw[->] (3) to (5);
		\draw[->] (3) to (6);
		\draw[loop left] (1) to node[left]{$f_1$} (1);
		\draw[loop left] (2) to node[left]{$f_k$} (2);
		\draw[loop right] (5) to node[right]{$g_1$} (5);
		\draw[loop right] (6) to node[right]{$g_l$} (6);
		\draw[loop above,red] (3) to node[above]{$h$} (3);
		\end{tikzpicture}
		\end{equation}
	where the $f_i:U_i \to U_i$ agree on pairwise intersections, the $g_i:V/W_i \to V/W_i$ induce endomorphisms on $V/(W_i + W_j)$ and the $f_i$ are compatible with the $g_j$
in the sense that they form commutative diagrams
\[
\xymatrix{
U_i\ar[d]_{f_i}\ar[r] & V \ar[r] & V/W_j\ar[d]^{g_j}\\
U_i\ar[r] & V \ar[r] & V/W_j\\
}
\]
	Then there exists an endomorphism $h:V \to V$ compatible with~\eqref{ref-6.25-123}.
\item If in \eqref{ref-6.25-123} both $(f_i)_i$ and $(g_j)_j$ are all zero then any $h$ is obtained as a composition
\[
V\r \coker(\oplus_i U_i\r V)\xrightarrow{h'} \ker(V\r \oplus_j V/W_j)\r V
\]
\end{enumerate}
	\end{lemmas}
	\begin{proof}
\begin{enumerate}
\item
	By distributivity, one has for every $1 \leq m < k$ that 
		\begin{equation}
		(\sum_{i=1}^m U_i) \cap U_{m+1}=\sum_{i=1}^m (U_i \cap U_{m+1})
		\end{equation} 
	and one constructs inductively a (necessarily unique) endomorphism $f$ of $\sum U_i$ compatible with the $f_i$. A similar reasoning holds for the $(V/W_j)_j$ 
(using Proposition \ref{ref-2.3.1-13} one may pass to subspaces $(V/W_j)^\ast=W^{\perp}_j\subset V^\ast$) and one obtains a 
	unique endomorphism $g$ of $V/\cap W_i$ compatible with the $g_i$ which fits in a diagram
		\begin{equation}
		\label{ref-6.27-124}
\xymatrix{
		\sum_i U_i \ar[d]_{f}\ar@{^(->}[r]^-i &V \ar@{->>}[r]^-p& V/\cap_j W_j\ar[d]^g\\
		\sum_i U_i \ar@{^(->}[r]_-i &V \ar@{->>}[r]_-p& V/\cap_j W_j
}
		\end{equation}
From the compatibility of $(f_i)_i$ and $(g_j)_j$ one obtains that this diagram is commutative. We have to complete it with a middle arrow $h:V\r V$. 
Elementary diagram chasing shows this is possible.
\item This is obvious.\qed
\end{enumerate}
\def\qed{}	\end{proof}

\begin{remarks}
  The classical example of three lines through the origin
  in the plane shows that the distributivity hypotheses are crucial.
\end{remarks}

\begin{corollarys}
\label{ref-6.2.6-125}
Let $\Lambda_1\subset\Lambda$ be a saturated subset and let $\Uscr_1$ be the corresponding full subcategory in $\Uscr$. Then the restriction map
\[
\End_{\Uscr}(M)\r \End_{\Uscr_1}(M)
\]
is surjective. 
\end{corollarys}
\begin{proof}
Choose saturated sets $\Lambda_1\subset\Lambda_2\subset\cdots\subset\Lambda$ such that $\Lambda=\bigcup_i\Lambda_i$ and such that the elements of
$\Lambda_{i+1}-\Lambda_i$ are incomparable for all $i$. Let $\Uscr_i$ be the corresponding full categories of $\Uscr$. Then
\[
\End_{\Uscr}(M)=\invlim_i \End_{\Uscr_i}(M)
\]
Since all the transition maps $\End_{\Uscr_{i+1}}(M)\r \End_{\Uscr_i}(M)$ are surjective by Theorem \ref{ref-6.2.3-120} the conclusion follows. 
\end{proof}
	\begin{propositions}
	\label{ref-6.2.7-126}
	For a given~$\lambda \in \Lambda$, the comodule $L(\lambda)=\im(\Delta(\lambda) \hookrightarrow M(\lambda) \twoheadrightarrow \nabla(\lambda))$
defined in \eqref{ref-6.5-104}
	is simple.
	\end{propositions}
	\begin{proof}
Let $\Lambda_1$, $\Lambda_2$ be a usual but assume in addition $\Lambda_2-\Lambda_1=\{\lambda\}$. By Corollary \ref{ref-6.2.6-125} it is sufficient to prove that $L(\lambda)$
is simple as $\End_{\Uscr_2}(M)$-module.

Since $L(\lambda)$ is an $\End_{\Uscr_2}(M)$-module we have an ``action map''
	$$
	\gamma:\End_{\Uscr_2}(M) \to \End_k(L(\lambda))
	$$
 To prove simplicity of $L(\lambda)$, we will show that $\gamma$
	is surjective. Consider the following diagram
	$$
	\begin{tikzcd}
	M(\lambda) \rar{h} \dar{\pi} & M(\lambda) \rar[equals] \dar{\pi} & M(\lambda) \\
	\nabla(\lambda) \rar{g} \dar[bend right=50,dashed,red,swap]{j'} & \nabla(\lambda) & \\
	L(\lambda) \uar[hook]{j} \rar{f} & L(\lambda) \uar[hook]{j} \rar[bend left=50,dashed,red]{p'} & \Delta(\lambda) 
	\lar{p}  \arrow{uu}[hookrightarrow]{i} \\
	\end{tikzcd}
	$$ 
	Here~$j,i,p$ and~$\pi$ are the obvious inclusion and quotient maps and the red arrows denote fixed~$k$-splittings. For an arbitrary~$f \in \End_k(L(\lambda))$, the linear maps~$g$ and~$h$ are defined 
	as follows:
	
	\[\lefteqn{
		\overbrace{
			\phantom{M(\lambda) \xrightarrow{\pi} \nabla(\lambda) \xrightarrow{j'} L(\lambda) \xrightarrow{f} 
			L(\lambda) \xrightarrow{p'} \Delta(\lambda) \xrightarrow{i} M(\lambda)}}
		^{h}} M(\lambda) \xrightarrow{\pi}
		\underbrace{\nabla(\lambda) \xrightarrow{j'} L(\lambda) \xrightarrow{f} L(\lambda) \xrightarrow{p'} 	
			\Delta(\lambda) \xrightarrow{i} M(\lambda) \xrightarrow{\pi} \nabla(\lambda)}_{g}\]
	
	In particular, the two squares on the left of the diagram are commutative. Also,~$h$ factorizes as~
	$M \xrightarrow{\pi} \nabla(\lambda) \to \Delta(\lambda) \xrightarrow{i} M(\lambda)$ and is thus in 
	$\End_{\Uscr_2}(M)$. From the commutativity it follows that~$h$ induces~$f$ so~$\gamma(h)=f$.
	\end{proof}
        Let $I:=\prod_{\lambda\in\Lambda_2-\Lambda_1} I_\lambda:=\prod_{\lambda\in\Lambda_2-\Lambda_1}
        \Hom_k(\nabla(\lambda),\Delta(\lambda))=\ker\beta$ be the ideal
        constructed in Theorem \ref{ref-6.2.3-120}.  

We will show
        that $I$ satisfies the conditions for being a heredity ideal
        as given in Definition \ref{ref-2.10.1-46} although we need not
        assume that our algebras are finite dimensional since they are
        pseudo-compact (see \S\ref{ref-2.9-39}). 
We denote by $\mod(\End_{\Uscr_i}(M))$ the category of discrete
finite dimensional left $\End_{\Uscr_i}(M)$-modules.
\begin{lemmas}
\label{ref-6.2.8-127}
The ideal~$I$ has an idempotent generator given by
$$
e=\prod_{\lambda\in \Lambda_2-\Lambda_1} e_{\lambda}:\nabla(\lambda) \xrightarrow{j'} L(\lambda) \xrightarrow{p'} \Delta(\lambda),
$$
where~$j'$ and~$p'$ are arbitrary splittings of the canonical inclusion and surjection. 
\end{lemmas}
\begin{proof}
Consider~$I$ as a ring without unit by using the canonical map
$$
c=\oplus c_{\lambda}: \Delta(\lambda) \to \nabla(\lambda),
$$
i.e. for~$f,g \in I$, the multiplication is defined by~$f \star g = f c g$.

Let~$f \in I$ be arbitrary. Since by defining~$e$ we fixed a splitting, it can be represented as 
$$
e: \nabla=L \oplus X \xrightarrow{\begin{pmatrix} \mathbf{1} & 0 \\ 0 & 0 \end{pmatrix}} L \oplus Y=\Delta,
$$
and similarly for~$c$.

Now it is easy to find some finite number of linear endomorphisms~$f_1, \ldots, f_n$ of~$\nabla$ and~$g_1, \ldots, g_n$ of~$\Delta$ such that
$$
f=g_1 \circ e \circ f_1 + \ldots + g_n \circ e \circ f_n.
$$
This is similar to proving that a matrix ring is simple. To finish the proof, it suffices to factor each of the~$f_i$ (respectively~$g_i$) as 
\begin{equation}
\begin{aligned}
f_i &= c \circ f_i^1 + \ldots + c \circ f_i^k, \\
g_i &=g_i^1 \circ c + \ldots g_i^l \circ c.
\end{aligned}
\end{equation}
Diagrammatically, this can be represented as:
$$
\begin{tikzcd}
M \rar{\pi} & \nabla \rar{f} \dar{\sum f_i} \dlar[swap]{\sum f_i^j} & \Delta \rar{i} & M \\
\Delta \rar[swap]{c} & \nabla \rar[swap]{e} & \Delta \uar{\sum g_i} \rar[swap]{c} & \nabla \ular[swap]{\sum g_i^j}
\end{tikzcd}
$$
where all diagrams commute.
\end{proof}

\begin{lemmas}
\label{ref-6.2.9-128}
If~$\rad_{\Uscr_2}(M)$ denotes the radical of~$\End_{\Uscr_2}(M)$, then
	\begin{equation}
	\label{ref-6.29-129}
	I \rad_{\Uscr_2}(M) I=0.
	\end{equation}
\end{lemmas}
\begin{proof}
Since we already know $I$ is idempotent,~\eqref{ref-6.29-129} follows if $e \big(\rad_{\Uscr_2}(M)\big) e=0$. Now by Proposition~\ref{ref-6.2.7-126}, it suffices to show that
$$
e \big( \cap_{\lambda\in \Lambda_2-\Lambda_1} \ann L(\lambda) \big) e=0,
$$
which is clear since~$L(\lambda)$ is by definition the image of~$c$.
\end{proof}

\begin{lemmas}
\label{ref-6.2.10-130}
The category $\mod(\End_{\Uscr_1}(M))$ is closed under extensions in $\mod(\End_{\Uscr_2}(M))$.
\end{lemmas}
\begin{proof}
This follows immediately from the exact sequence
$$
0 \to I \to \End_{\Uscr_2}(M) \to \End_{\Uscr_1}(M) \to 0,
$$
since $I=I^2$.
\end{proof}

\begin{lemmas}
\label{ref-6.2.11-131}
For $\lambda\in \Lambda_2-\Lambda_1$, the modules~$\Delta(\lambda)$ and~$\nabla(\lambda)$ are projective, respectively injective in $\mod(\End_{\Uscr_2}(M))$.
\end{lemmas}
\begin{proof}
We only prove that~$\Delta(\lambda)$ is projective. The proof of injectivity for~$\nabla(\lambda)$ is similar. One has~$I_{\lambda} \cdot \Delta(\lambda)=\Delta(\lambda)$: this is linear algebra just like in Lemma~\ref{ref-6.2.8-127}. In particular~$I_{\lambda} \cdot L(\lambda)=L(\lambda)$ since~$L(\lambda)$ is a quotient of~$\Delta(\lambda)$.

Also~$\Delta(\lambda)':=\Ker(\Delta(\lambda) \to L(\lambda)) \in \mod(\End_{\Uscr_1}(M))$ since
\begin{equation}
\begin{aligned}
\Ker(\Delta(\lambda) \to L(\lambda)) &= \Ker(\Delta(\lambda) \to \nabla(\lambda)) \\
&\subset \Ker(M(\lambda) \to \nabla(\lambda)),
\end{aligned}
\end{equation}
and~$I_{\lambda}$ clearly works trivially on this last module. In particular,  
\begin{equation}
\label{ref-6.31-132}
\Hom_{\End_{\Uscr_2}(M)}(\Delta(\lambda)',L(\lambda))=0,
\end{equation}
since otherwise we would get a surjective map~$0=I_{\lambda}\cdot \Delta' \to I_{\lambda} \cdot L(\lambda)=L(\lambda)$.

We will now use this to show that~$\Delta(\lambda)$ is Schurian. By composing with the quotient map, any element~$u \in \End_{\End_{\Uscr_2}(M)}(\Delta(\lambda))$ induces a morphism~$\Delta(\lambda) \to L(\lambda)$, which by~\eqref{ref-6.31-132} descends to an endomorphism of~$L(\lambda)$, given by a scalar~$c$. Now consider the endomorphism~$u-c:\Delta(\lambda) \to \Delta(\lambda)$, which defines a morphism~$\Delta(\lambda) \to \Delta(\lambda)'$. Now since~$I_{\lambda} \cdot \Delta(\lambda)=\Delta(\lambda)$ and~$I_{\lambda} \cdot \Delta(\lambda)'=0$ by the previous paragraph, the image of the endomorphism~$u-c$ is zero. 

Now we show that any surjective map~$p:M \to \Delta(\lambda)$ splits. Since~$I_{\lambda} \Delta(\lambda)=\Delta(\lambda)$, we can assume without loss of generality that~$I_{\lambda}M=M$. Now take
$$
m=\sum_{i=1}^u \phi_i m_i \notin \Ker(p),
$$
with~$\phi_i \in I_{\lambda}, m_i \in M$. Then~$m$ is in the image of
$$
I_{\lambda}^u \to M:(\phi_i)_i \mapsto \sum_i \phi_i m_i,
$$
and since~$I_{\lambda} \cong \Delta(\lambda)^v$ for some power~$v$, this gives a map~$q:\Delta(\lambda)^w \to M$. By choice of~$m$ the composition~$\Delta(\lambda)^w \to M \to \Delta(\lambda)$ is non-zero and since~$\Delta(\lambda)$ is Schurian, one of the components~$q_i:\Delta(\lambda) \to M$ of~$q$ has to give a non-zero scalar after the composition~$\Delta(\lambda) \to M \to \Delta(\lambda)$. This~$q_i$ is the required splitting. 
\end{proof}

\subsection{Statement and proof of the main theorem}

	\begin{theorems}
	\label{ref-6.3.1-133}
Let $\Lambda_1\subset \Lambda$ be a saturated subset with associated full subcategory $\Uscr_1\subset \Uscr$.
		\begin{enumerate}
		\item the coalgebra $\coend_{\Uscr_1}(M)$ is quasi-hereditary, with respect to the partially ordered set $(\Lambda_1,\le)$,
		\item the simple, standard and costandard comodules are given by~\eqref{ref-6.5-104}\eqref{ref-6.4-103}, and~\eqref{ref-6.3-102}
for $\lambda\in \Lambda_1$.
		\end{enumerate}
	\end{theorems}
Specialising to $\Lambda_1=\Lambda$ and using Theorem \ref{ref-5.1-92} we obtain
\begin{corollarys}
\label{ref-6.3.2-134}
	For any Koszul Artin-Schelter regular algebra $A$ of global dimension $d$, the following statements hold:
		\begin{enumerate}
		\item the coalgebra $\uaut(M)$ is quasi-hereditary, with respect to the partially ordered set $(\Lambda,\le)$,
		\item the simple, standard and costandard comodules are given by~\eqref{ref-6.5-104}\eqref{ref-6.4-103}, and~\eqref{ref-6.3-102}.
		\end{enumerate}
\end{corollarys}
	\begin{proof}[Proof of Theorem \ref{ref-6.3.1-133}]
We first show that $\coend_{\Uscr_1}(M)$ is quasi-hereditary
	according to Definition~\ref{ref-2.11.1-56}. Choose a total ordering $\lambda_1<'\lambda_2<'\cdots$ on $\Lambda_1$ refining $<$. Let
$\Uscr^{(i)}=\{\lambda_1,\ldots,\lambda_i\}\subset \Uscr$ and $\End_n(M):=\End_{\Uscr^{(n)}}(M)$.

By Theorem~\ref{ref-6.2.3-120}, the coalgebra $\End_{\Uscr_1}(M)^*$ admits an exhaustive filtration
		\begin{equation}
		0 \subset \End_1^*(M) \subset \End_2^*(M) \subset \cdots \subset \End_{\Uscr_1}^*(M)  \subset \cdots  
		\end{equation}
	of finite dimensional subcoalgebras, and we claim that this filtration is a heredity cochain. To this end, we need to check that for every $i$
		\begin{equation}
		0=(\End^*_i(M)/\End^*_i(M))^* \subset (\End^*_i(M)/\End^*_{i-1}(M))^* \subset \cdots \subset \End_i(M)
		\end{equation}
	is a heredity chain. Now since 
	$$
	(\End^*_i(M)/\End^*_{i-t}(M))^*\cong \ker(\End_i(M) \to \End_{i-t}(M)),
	$$ 
	we need to check that 
		\begin{equation}
		\frac{\ker(\End_i(M) \to \End_{i-t-1}(M))}{\ker(\End_i(M) \to \End_{i-t}(M))} \vartriangleleft \frac{\End_i(M)}{\ker(\End_i(M) \to \End_{i-t}(M))}
		\end{equation}
	is a heredity ideal, or equivalently, that $I_{i-t} \vartriangleleft \End_{i-t}(M)$ is.
	
	By Lemma~\ref{ref-6.2.8-127}, $I_{i-t}$ is generated by an idempotent and by Lemma~\ref{ref-6.2.9-128} we know that~$I_{i-t} \rad_{i-t}(M) I_{i-t}=0$. Now since
		\begin{equation}
		I_{\lambda} \cong \Delta(\lambda)^{\oplus \dim_k \nabla(\lambda)},
		\end{equation}
	as one-sided module, we know that~$I_{\lambda}$ is also projective as a one-sided module by Lemma~\ref{ref-6.2.11-131} and we are done.

It follows from	Proposition~\ref{ref-2.11.5-61} and the specific form of the heredity cochain that $\coend_{\Uscr_1}(M)$ is quasi-hereditary with respect
to $(\Lambda,\le')$ with simples, standard, costandard comodules given by~\eqref{ref-6.5-104}\eqref{ref-6.4-103}, and~\eqref{ref-6.3-102}.

The fact that $\End_{\Uscr_1}(M)$ is quasi-hereditary with respect to $<$ follows from Lemma \ref{ref-2.10.7-52}.
	\end{proof}
	\begin{corollarys}
	\label{ref-6.3.3-135}
	Let $A$ be as above, then
	\begin{enumerate}
		\item For all $\lambda,\mu\in \Lambda$ \label{ref-1-136}
				\begin{align}
				{}[M(\lambda):\nabla(\mu)]&=|\Uscr_{\u}(\mu,\lambda)|\label{ref-6.36-137}\\
				[M(\lambda):\Delta(\mu)]&=|\Uscr_{\d}(\lambda,\mu)|\label{ref-6.37-138}
				\end{align}
		\item Let $G_0(\uaut(A))$ be the representation ring of $\uaut(A)$. There is a ring isomorphism 
			\begin{equation}
			\mathbb{Z} \langle r_1,\ldots,r_{d-1},r_d^{\pm 1} \rangle \to G_0(\uaut(A)): r_i \mapsto [R_i]. 
			\end{equation}
	\end{enumerate}
	\end{corollarys}
	\begin{proof}
	For the first statement, evaluating the resolution~\eqref{ref-3.12-75} for $S_u$ at $v$ we obtain exact sequences
		\begin{multline*}
		\cdots \r\bigoplus_{i_1<\cdots<i_n} k\Uscr_{\u}(v,{\operatorname{source} (f_{i_1}\times_u \dots \times_u f_{i_n})})\r \cdots \\
		\r \bigoplus_{i<j} k\Uscr_{\u}(v,{\operatorname{source}(f_i\times_x f_j)})\r \bigoplus_{f_i:u_i\r u\in \Uscr_{\u,1}}k\Uscr_{\u}(v,u_i)\r k\Uscr_{\u}(v,u)\r k\delta_{u,v}\r 0
		\end{multline*}
	which gives us a recursion relation between $|\Uscr_{\u}(v,u)|$ and $|\Uscr_{\u}(v,u')|$ for $u'<u$. Since $|\{u'|u'<u\}|$ is finite this
	recursion relation completely determines $|\Uscr_{\u}(-,-)|$.

	By Lemma~\ref{ref-4.4-89} there are similar exact sequences of the form: 
		\begin{multline}
		\label{ref-6.39-139}
		\cdots \r\bigoplus_{i_1<\cdots<i_n} M({\operatorname{source} (f_{i_1}\times_u \dots \times_u f_{i_n})}) \r \cdots \\
		\cdots \r \bigoplus_{i<j} M({\operatorname{source}(f_i\times_x f_j)})\r \bigoplus_{f_i:u_i\r u\in U_{\u,1}}M(u_i)\r M(u)\r \nabla(u)\r 0
		\end{multline}
	where   by Corollary \ref{ref-4.2-86}.
		\[
		 M({\operatorname{source} (f_{i_1}\times_u \dots \times_u f_{i_n})})=\bigcap_i M(u_i).
		\]
	As $\uaut(A)$ is quasi-hereditary we have $\Ext^i_{\uaut(A)}(\Delta(v),\nabla(u))=k\delta_{i,0}$ (see Proposition \ref{ref-2.10.8-53}(2)) and by hence by Corollary \ref{ref-6.1.6-111}(1):

   $\Ext^i_{\uaut(A)}(\Delta(v),M(u))=0$
for $i>0$. Thus~\eqref{ref-6.39-139} remains exact after 
	applying $\Hom_{\uaut(A)}(\Delta(v),-)$.

From Corollary \ref{ref-6.1.6-111}(1) and Proposition \ref{ref-2.10.8-53}(2) we also deduce the formula
		\[
		[M(u)	:\nabla(v)]=\dim \Hom_{\uaut(A)}(\Delta(v),M(u)).
		\]
	and we see that $\dim\Hom_{\uaut(A)}(\Delta(v),M(u))$ satisfies exactly the same recursion relation
	as $|\Uscr_{\u}(v,u)|$. The equality \eqref{ref-6.36-137} follows. The proof of \eqref{ref-6.37-138} is the same.
	
For the second statement note that 
since in  a quasi-hereditary (co)algebra $C$, the costandard comodules are related
by a triangular matrix to the simple comodules it is clear that they
form a $\mathbb{Z}$-basis of $G_0(C)$. Using that $M(\lambda) \in \Fscr(\nabla)$, and that 
according to \eqref{ref-6.36-137}
the $M(\lambda)$ are 
	related to the costandard comodules by a unitriangular matrix,  it is clear that the $[M(\lambda)]$, for
	$\lambda \in \Lambda$ form a basis of $G_0(\uaut(A))$. It now suffices to note that the functor $M$ is monoidal and maps $r_i$ to $R_i$. This proves the last statement.
	\end{proof}
\begin{theorems} \label{ref-6.3.4-140}
Let $\Lambda_1\subset \Lambda$ be a saturated subset and let $\Uscr_1$ be the corresponding full subcategory of $\Uscr$. Then the restriction map 
\begin{equation}
\label{ref-6.40-141}
\coend_{\Uscr_1}(M)\r \coend_{\Uscr}(M)
\end{equation}
identifies
$
\coend_{\Uscr_1}(M)
$
with $(\coend_{\Uscr}(M))(\Lambda_1)$ (see \S\ref{ref-2.11-55} for notation). 
\end{theorems}
\begin{proof}
We consider the dual morphism
\begin{equation}
\label{ref-6.41-142}
\End_{\Uscr}(M)\r \End_{\Uscr_1}(M)
\end{equation}
We have already shown it is surjective in Corollary \ref{ref-6.2.6-125} which implies in particular that \eqref{ref-6.40-141} is injective. 

We now claim that $\mod(\End_{\Uscr_1}(M))$ is closed under extensions. Consider an exact sequence 
\[
0\r A\r B\r C\r 0
\]
with $A,C\in \mod(\End_{\Uscr_1}(M))$ and $B\in \mod(\End_{\Uscr}(M))$. Using the fact that $A,B,C$ are discrete and the nature of the topology on $\End_{\Uscr_1}(M)$ (see \eqref{ref-2.16-40}) there exists
a finite saturated subset $\Lambda_2\subset \Lambda$ with $\Lambda_3:=\Lambda_2\cap\Lambda_1$ and associated full subcategories $\Uscr_i\subset\Uscr$
such that $B\in \mod(\End_{\Uscr_2}(M))$ and $A,C\in \mod(\End_{\Uscr_3}(M))$.

By repeatedly applying \eqref{ref-6.23-121} and Lemma \ref{ref-6.2.10-130} we obtain $B\in \mod(\End_{\Uscr_3}(M))$ and thus $B\in\mod(\End_{\Uscr_1}(M))$. So $\mod(\End_{\Uscr_1}(M))$ is indeed closed under extensions.
According to Theorem \ref{ref-6.3.1-133} the simple objects in  $\mod(\End_{\Uscr_1}(M))$ are given by $L(\lambda)$ for $\lambda\in \Lambda_1$. It is easy to see that this
implies $(\coend_{\Uscr}(M))(\Lambda_1)=\coend_{\Uscr_1}(M)$.
\end{proof}
\begin{corollarys} 
\label{ref-6.3.5-143} We have
\[
\uend(A)\cong \uaut(A)(\Lambda^+)
\]
\end{corollarys}
\begin{proof} This follows immediately from Theorem \ref{ref-6.3.4-140} with
$\Lambda_1=\Lambda^+$.
\end{proof}

\section{Schur-Weyl duality and derived Tannaka-Krein}
\label{ref-7-144}
\subsection{Schur-Weyl duality}
Consider a pair $(\Cscr,F)$, where $\Cscr$ is a category and $F:\Cscr \to \Vect$ is a functor, just like in Section~\ref{ref-2.9-39}. We want to argue that intermediate properties of the associated evaluation functor $\eval_F:\Cscr \to \comod(\coend(F))$  can be very useful for studying the representations of $\coend(F)$. In particular we propose the following definition.

	\begin{definitions}
	\label{ref-7.1.1-145}
	The pair $(\Cscr,F)$ satisfies Schur-Weyl duality if $\eval_F$ is full. It satisfies strong Schur-Weyl duality if $\eval_F$ is fully faithful.
	\end{definitions}
	
	\begin{examples}
	The definition is inspired by the classical case: let $V$ denote a finite dimensional vector space over a field of characteristic zero, and $kS_n$ the group ring of the symmetric group. Schur-Weyl 
	duality says that the natural algebra morphism
		\begin{equation}
		\label{ref-7.1-146}
		kS_n \to \End_{\GL(V)}(V^{\otimes n})
		\end{equation}
	is surjective for any $n$. This phrasing highlights the importance of the double commutant theorem. 
	
	Take $\Cscr$ to be the free $k$-linear symmetric strict monoidal category on 1 object $v$ and let $F$ be defined by $F(v)=V$. Then one easily checks that surjectivity of~\eqref{ref-7.1-146} is equivalent to 
	surjectivity of
		\begin{equation}
		\Hom_{\Cscr}(v^{\ot n},v^{\ot n}) \to \Hom_{\coend(F)}(V^{\ot n},V^{\ot n}),
		\end{equation}  
	so Definition~\ref{ref-7.1.1-145} is satisfied.
	\end{examples}
Taking into account that $\uaut(A)=\coend_{\Uscr}(M)$ by Theorem \ref{ref-5.1-92}, the following is
a version of Schur-Weyl duality in our setting.
	\begin{theorems}[Schur-Weyl duality]
\label{ref-7.1.3-147}
	\label{ref-7.1.3-148} 
	The linearised functor 
		\begin{equation}
		\eval_M:k\Uscr \r \comod(\uaut(A))
		\end{equation} 
	is fully faithful, i.e. $(k\Uscr,M)$ satisfies strong Schur-Weyl duality.
	\end{theorems}
	\begin{proof}
	Using the fact that $M(\lambda)\in \Fscr(\Delta)\cap \Fscr(\nabla)$ we find
		\begin{align*}
		\dim \Hom_{\uaut(A)}(M(\mu),M(\nu))&=\sum_{\lambda \in \Lambda} [M(\mu):\Delta(\lambda)][M(\nu):\nabla(\lambda)]\\
		&=\sum_{\lambda \in \Lambda} |\Uscr_{\d}(\mu,\lambda)|\cdot |\Uscr_{\u}(\lambda,\nu)|\\
		&=|\Uscr(\mu,\nu)|
		\end{align*}
where the first equality follows from Corollary \ref{ref-6.3.3-135} combined with the orthogonality of $\Delta$ and $\nabla$, the second
equality follows from \eqref{ref-6.36-137}\eqref{ref-6.37-138} and
	 the last equality follows from Proposition~\ref{ref-3.3.1-80}. 

Hence it is sufficient to prove that $k\Uscr(\mu,\nu)\r \Hom_{\uaut(A)}(M(\mu),M(\nu))$ is surjective.
	Using the resolutions \eqref{ref-6.39-139} for $\Delta(\nu)$ and the dual coresolutions for $\nabla(\mu)$ we get an exact sequence
		\begin{multline*}
		\bigoplus_{\substack{ f_i:\nu_i\r \nu \\ \in \Uscr_{\u,1}}}\Hom_{\uaut(A)}(M(\mu),M(\nu_i))
		\oplus
		\bigoplus_{\substack{g_i:\mu\r \mu_i \\ \in \Uscr_{\d,1}}}\Hom_{\uaut(A)}(M(\mu_i),M(\nu))\\
		\r \Hom_{\uaut(A)}(M(\mu),M(\nu))\r k\delta_{\mu,\nu}\r 0
		\end{multline*}
	It follows that every $p:M(\mu)\r M(\nu)$ can be written as a sum of maps which factor through some $M(\nu_i)$ or $M(\mu_i)$ 
	together with a scalar multiple of the identity morphism (if $\nu=\mu$). The surjectivity claim now follows by induction.
	\end{proof}
	
	\subsection{Derived Tannaka-Krein}
\label{ref-7.2-149}

	Let $\perf(\Uscr^\circ)$
be the triangulated category of finite complexes of finitely generated
projective right $\Uscr$-modules.
	One may make $\perf(\Uscr^\circ)$ into a triangulated monoidal category by putting
	\[
	k\Uscr(-,u)\otimes k\Uscr(-,v)=k\Uscr(-,uv)
	\]
	and extending to complexes. The functor $M$ extends to an exact monoidal functor
	\begin{equation}
	\label{ref-7.4-150}
	M:\perf(\Uscr^\circ)\r D^b({\uaut(A)}):\Uscr(-,u)\mapsto M(u)
	\end{equation}
	where we have written $D^b(\uaut(A))$ for $D^b(\comod(\uaut(A)))$. At the risk of
	confusing various tensor products the functor $M$ can be written
	intrinsically as $-\Lotimes_{\Uscr}M$. 
	
	\begin{corollarys}[Derived Tannaka-Krein]
\label{ref-7.2.1-151}
	The functor~\eqref{ref-7.4-150} is an equivalence of monoidal triangulated categories. 
	\end{corollarys}
	\begin{proof} It follows from Theorem~\ref{ref-7.1.3-148} and induction over distinguished triangles
	that \eqref{ref-7.4-150} is fully faithful. By the theory of quasi-hereditary coalgebras we know
	that $\nabla(u)_u$ generates $D^b({\uaut(A)})$. From \eqref{ref-6.36-137} we deduce $[M(u):\nabla(u)]=1$ and all other
	composition factors $\nabla(v)$ satisfy $v<u$. By induction it then follows that $D^b({\uaut(A)})$ is generated
	by $M(u)_u$ as well. From this it follows that \eqref{ref-7.4-150} is essentially surjective. 
	\end{proof}
\begin{remarks} 
\label{ref-7.2.2-152}
Since by Corollary \ref{ref-6.1.6-111} the $M(\lambda)$ are partial tilting modules 
(see \S\ref{ref-2.10-45}) for $\comod(\uaut(A))$, Theorem \ref{ref-7.1.3-147} can also be expressed as saying that $k\Uscr$ is Morita equivalent to a suitably defined categorical version of the Ringel dual (see \ref{ref-2.10-45}) of $\uaut(A)$.
Corollary \ref{ref-7.2.1-151} is then expressing the fact that Ringel duality
yields a derived equivalence.
\end{remarks}
The exact sequences \eqref{ref-6.39-139} and their dual versions allow us to obtain the following result.
\begin{theorems}
\label{ref-7.2.3-153}
 $\comod({\uaut(A)})$ as a monoidal category depends only on the global dimension of $A$.
\end{theorems}
\begin{proof} Let $M$ be as in \eqref{ref-7.4-150} and put
\begin{align*}
\Delta_{\Uscr}(u)&=M^{-1}(\Delta(u))\in \perf(\Uscr^\circ)\\
\nabla_{\Uscr}(u)&=M^{-1}(\nabla(u))\in \perf(\Uscr^\circ)
\end{align*}
From the explicit resolution \eqref{ref-6.39-139} and its dual version we see that
$\Delta_{\Uscr}(u)$ and $\nabla_{\Uscr}(u)$ only depend on $\Uscr$.

Let $\Dscr=D^b({\uaut(A)})$ and let $(\Dscr_{\le 0},\Dscr_{\ge 0})$ be its natural $t$-structure. Since
$\Delta(u)\twoheadrightarrow
L(u)\hookrightarrow \nabla(u)$ it is easy to see that we have
\begin{align*}
\Dscr_{\ge 0}&=\{P\in \Dscr\mid \forall u\in \Lambda:\forall i<0:\Ext^i_{\uaut(A)}(\Delta(u),P)=0\}\\
\Dscr_{\le 0}&=\{P\in \Dscr\mid \forall u\in \Lambda:\forall i<0:\Ext^i_{\uaut(A)}(P,\nabla(u))=0\}
\end{align*}
Using the derived equivalence $M$ we may transfer this $t$-structure to a $t$-structure on $\perf(\Uscr^\circ)$. We find a monoidal equivalence
\begin{multline*}
\comod({\uaut(A)})\cong \{P\in \perf(\Uscr^\circ)\mid \forall u\in \Lambda:\forall i<0: \\ \Ext^i_{\uaut(A)}(\Delta_{\Uscr}(u),P)=0,\Ext^i_{\uaut(A)}(P,\nabla_{\Uscr}(u))=0\}\\
\end{multline*}
and the righthandside only depends on $\Uscr$.
\end{proof}
\subsection{Comparison with the results in \cite{kriegk-vandenbergh} for $\uend(A)$}
\label{ref-7.3-154}
By Theorem \ref{ref-6.3.1-133} and Theorem \ref{ref-5.1-92} we know that $\uend(A)$ is quasi-hereditary with $\Delta(\lambda)=M(\lambda)$. Hence by 
\eqref{ref-6.5-104} $L(\lambda)=\nabla(\lambda)$. Using Proposition \ref{ref-2.10.8-53}(2)
one obtains that 
$M(\lambda)=\Delta(\lambda)$ is projective. Hence the $(\Delta(\lambda) )_{\lambda\in \Lambda^+}$ form a system of projective generators for 
$\comod(\uend(A))$ and we conclude by \eqref{ref-7.1.3-148} that there is an equivalence of categories
\[
M^+:\mod(\Uscr_{\u}^{+,\circ})\r \comod(\uend(A)):\Uscr^+_{\u}(-,u)\mapsto M^+(u)
\]

\begin{remarks} Although in this section we have kept our blanket assumption that $A$ is Koszul Artin-Schelter
  regular of global dimension $d$ the results generalise
  with little modification to the case that $A$ is just Koszul if we
  replace $\Uscr_{\u}^+$ by $\Uscr_{\u,\infty}^+$ (see
  \S\ref{ref-2.6-24}). The only additional ingredient is that Proposition
  \ref{ref-6.1.7-114} is no longer true. So one has to replace $\Lambda_{\infty}^+$ by a smaller poset corresponding to the non-zero $\nabla(\lambda)=L(\lambda)$.
This poset depends only on the Hilbert series of $A$ since the $\nabla(\lambda)=0$  is equivalent to $\dim \nabla(\lambda)=0$ and the dimension of $\nabla(\lambda)$ can be computed using the standard resolution \eqref{ref-6.39-139}. In this way one recovers all the results in \cite{kriegk-vandenbergh}.
\end{remarks}

\appendix
\label{ref-7.3-155}
\section{Explicit presentations}
\label{ref-A-156}
In this appendix we discuss a presentation of $\uaut(A)$ in the case
that $A=k[x_1,\ldots,x_d]$.

If $F:\Cscr\r \Vect$ is a monoidal functor then a presentation of
$C=\coend_{\Cscr}(F)$ can be obtained directly from a presentation of
the underlying category $\Cscr$. Roughly speaking generators of $C$ as
an algebra correspond to generating objects in $\Cscr$ and generating
relations correspond to generating morphisms in $\Cscr$. In other
words by writing an arbitrary bialgebra $C$ in the form $\coend_{\Cscr}(F)$ we reduce its conceptual
complexity by one degree!

We now discuss this more precisely.  Adding isomorphisms if necessary we may assume that $\Ob(\Cscr)$ is a free monoid generated by a set $(X_k)_k$. Choose
bases~$e_{ki}$ for each $F(X_k)$. Then the corresponding ``matrix coefficients'' $(z_{kij})_{kij}\in C$ are defined via the coaction
\begin{equation}
\label{ref-A.1-157}
\delta(e_{ki})=\sum_j z_{kij} \otimes e_{kj}
\end{equation}
The matrix coefficients generate $C$ as an algebra, and they determine the coalgebra structure
via the following formulas
\begin{align*}
\Delta(z_{kij})&=\sum_p z_{kip}\otimes z_{pj}\\
\epsilon(z_{kij})&=\delta_{ij}
\end{align*}
Writing out the compatibility of \eqref{ref-A.1-157} with a morphism
in $\Cscr$ yields relations among the $(z_{kij})_{kij}$ and to obtain
a presentation of $C$ is its sufficient to do this for a set of generating morphisms.
Note that the relations among the morphisms in $\Cscr$ play no
role.
\begin{remark} Note however that the relations in $\Cscr$ play a vital role if one wants to use $\Cscr$ to 
derive properties of $\comod(C)$.
\end{remark}
When $A$ is as usual a Koszul Artin-Schelter regular algebra of global dimension $d$ then using a similar argument as in the proof of Theorem \ref{ref-5.1-92} one has the following economic presentation
\begin{equation}
\label{ref-A.2-158}
		\uaut(A)=\coend\bigg(\langle r_1, \ldots, r_{d-1},r_d^{\pm 1}  |
r_i \to r_1^i, r_a r_d^{-1} r_{d-a} \to 1 \rangle, M \bigg),
\end{equation}

To apply this with $A=k[x_1,\ldots,x_d]$ we need the concept of a Manin matrix \cite{chevrov-falqui-rubtsov}.
A $2\times 2$-Manin matrix is a matrix of the form
\begin{equation}
\label{ref-A.3-159}
\begin{pmatrix}
a&b\\
c&d
\end{pmatrix}
\end{equation}
satisfying the relations $ac=ca$, $db=db$, $ad-cb=da-bc$. A $d\times d$-Manin
is a matrix $X=(x_{ij})_{i,j=1,\ldots,d}$ such that every $2\times 2$-submatrix is a 
$2\times2$-Manin matrix. The Manin determinant of a Manin matrix is given by
\[
|X|:=\sum_{\sigma\in S_d} (-1)^{\sigma} x_{1\sigma(1)}\cdots x_{d\sigma(d)}
\]
The Manin determinant behaves correctly when exchanging rows or columns
and moreover it is zero if there are duplicate rows or columns.

We write
$X_{i_1,\ldots,i_p,j_1,\ldots,j_p}$
for the matrix $(x_{i_p,j_q})_{pq}$.
Using \eqref{ref-A.2-158} we find 
\begin{lemma}
 $\uaut(A)$ is generated by the entries of a generic Manin
matrix~$X$ with $|X|$ formally inverted satisfying the following additional relations for $b=1,\ldots,d-1$.
\begin{equation}
\label{ref-A.4-160}
\epsilon(\sigma)=
\sum_{\tau} (-1)^{\tau}
|X_{\sigma(1)\cdots \sigma(b),\tau(1)\cdots \tau(b)}|\delta^{-1}
|X_{\sigma(b+1)\cdots \sigma(d),\tau(b+1)\cdots \tau(d)}|
\end{equation}
where
\begin{enumerate}
\item $\delta=|X|$.
\item $\sigma$ is a map $\{1,\ldots,d\}\r \{1,\ldots,d\}$ which is ascending
on $\{1,\ldots,b\}$ and $\{b+1,\dots,d\}$.
\item If $\sigma$ is a permutation then $\epsilon(\sigma)$ is its sign. Otherwise $\epsilon(\sigma)=0$.
\item $\tau$ is a permutation of $\{1,\ldots,d\}$ which is ascending
on $\{1,\ldots,b\}$, $\{b+1,\ldots,d\}$.
\end{enumerate}
\end{lemma}
It is not clear to us how many of the equations in \eqref{ref-A.4-160} are independent. 
\begin{example} If $d=2$ then $\aut(k[x,y])$ is generated by $a,b,c,d,\delta^{-1}$ with the following presentation
\begin{equation}
\label{ref-A.5-161}
\begin{aligned}
ac-ca & =  0 =  bd-db \\ 
ad-cb & =  \delta  =  da-bc\\
\delta \delta^{-1} & =  1  =  \delta^{-1}\delta,\\ 
a\delta^{-1}d-b\delta^{-1}c & =  1  =  d\delta^{-1}a-c\delta^{-1}b,\\
b\delta^{-1}a-a\delta^{-1}b & =  0  =  c\delta^{-1}d-d\delta^{-1}c
\end{aligned}
\end{equation}
\end{example}

%\bibliography{bib}

\begin{thebibliography}{10}

\bibitem{AS}
M.~Artin and W.~Schelter, \emph{Graded algebras of global dimension 3}, Adv. in
  Math. \textbf{66} (1987), 171--216.

\bibitem{berger}
R.~Berger, \emph{Confluence and {K}oszulity}, J. Algebra \textbf{201} (1998),
  no.~1, 243--283.

\bibitem{bichon}
J.~Bichon, \emph{Co-representation theory of universal co-sovereign {H}opf
  algebras}, J. Lond. Math. Soc. (2) \textbf{75} (2007), no.~1, 83--98.

\bibitem{bichon-dubois-violette}
J.~Bichon and M.~Dubois-Violette, \emph{The quantum group of a preregular
  multilinear form}, Lett. Math. Phys. \textbf{103} (2013), no.~4, 455--468.

\bibitem{chevrov-falqui-rubtsov}
A.~Chervov, G.~Falqui, and V.~Rubtsov, \emph{Algebraic properties of {M}anin
  matrices. {I}}, Adv. in Appl. Math. \textbf{43} (2009), 239--315.

\bibitem{chirvasitu}
A.~Chirvasitu, \emph{Grothendieck rings of universal quantum groups}, J.
  Algebra \textbf{349} (2012), 80--97.

\bibitem{cline-parshall-scott}
E.~Cline, B.~Parshall, and L.~Scott, \emph{Finite-dimensional algebras and
  highest weight categories}, J. Reine Angew. Math. \textbf{391} (1988),
  85--99.

\bibitem{deng-du-parshall-wang}
B.~Deng, J.~Du, B.~Parshall, and J.~Wang, \emph{Finite dimensional algebras and
  quantum groups}, Mathematical Surveys and Monographs, vol. 150, American
  Mathematical Society, Providence, RI, 2008.

\bibitem{dlab-ringel-1}
V.~Dlab and C.M. Ringel, \emph{Quasi-hereditary algebras}, Illinois J. Math.
  \textbf{33} (1989), 280--291.

\bibitem{dlab-ringel-2}
\bysame, \emph{The module theoretical approach to quasi-hereditary algebras},
  Representations of Algebras and Related Topics, London Math. Soc. Lecture
  Note Series, vol. 168, Cambridge University Press, Cambridge, 1992,
  pp.~200--224.

\bibitem{donkin3}
S.~Donkin, \emph{Rational representations of algebraic groups}, Lecture Notes
  in Mathematics, vol. 1140, Springer-Verlag, Berlin, 1985, Tensor products and
  filtration.

\bibitem{donkin2}
\bysame, \emph{The {$q$}-{S}chur algebra}, London Mathematical Society Lecture
  Note Series, vol. 253, Cambridge University Press, Cambridge, 1998.

\bibitem{donkin}
\bysame, \emph{Tilting modules for algebraic groups and finite dimensional
  algebras}, Handbook of tilting theory, London Math. Soc. Lecture Note Ser.,
  vol. 332, Cambridge Univ. Press, Cambridge, 2007, pp.~215--257.

\bibitem{gabriel}
P.~Gabriel, \emph{Des cat\'egories ab\'eliennes}, Bull. Soc. Math. France
  \textbf{90} (1962), 323--448.

\bibitem{hai-lorenz}
P.~H. Hai and M.~Lorenz, \emph{Koszul algebras and the quantum {M}ac{M}ahon
  master theorem}, Bull. Lond. Math. Soc. \textbf{39} (2007), no.~4, 667--676.

\bibitem{he-torrecillas-vanoystaeyen-zhang}
J.W. He, B.~Torrecillas, F.~Van Oystaeyen, and Y.~Zhang, \emph{Dualizing
  complexes of noetherian complete algebras via coalgebras}, Comm. Algebra
  \textbf{42} (2014), no.~1, 271--285.

\bibitem{kriegk-vandenbergh}
B.~Kriegk and M.~Van~den Bergh, \emph{Representations of noncommutative quantum
  groups}, Proc. Lond. Math. Soc. (3) \textbf{110} (2015), no.~1, 57--82.

\bibitem{manin}
Y.~I. Manin, \emph{Quantum groups and non-commutative geometry}, Tech. report,
  Centre de Recherches Math\'ematiques, Universit\'e de Montreal, 1988.

\bibitem{mathieu}
O.~Mathieu, \emph{Filtrations de ${G}$-modules}, C. R. Acad. Sci. Paris S\'er.
  I Math. \textbf{309} (1989), no.~6, 273--276.

\bibitem{mitchell}
B.~Mitchell, \emph{Rings with several objects}, Advances in Math. \textbf{8}
  (1972), 1--161.

\bibitem{Ohn}
C.~Ohn, \emph{Quantum {${\rm SL}(3,{\Bbb C})$}'s: the missing case}, Hopf
  algebras in noncommutative geometry and physics, Lecture Notes in Pure and
  Appl. Math., vol. 239, Dekker, New York, 2005, pp.~245--255.

\bibitem{pareigis-2}
B.~Pareigis, \emph{Lectures on quantum groups and noncommutative geometry},
  Lecture Notes TU Munich, URL:
  http://www.mathematik.uni-muenchen.de/\~{}pareigis/Vorlesungen/02SS/QGand\-NCG.pdf.

\bibitem{pareigis-1}
\bysame, \emph{Endomorphism bialgebras of diagrams and of non-commutative
  algebras and spaces}, Advances in Hopf Algebras, LN Pure and Appl. Math.,
  vol. 158, Marcel Dekker, Inc., 1994, pp.~153--186.

\bibitem{parshall-scott-wang}
B.~Parshall, L.~Scott, and J.-P. Wang, \emph{Borel subalgebras redux with
  examples from algebraic and quantum groups}, Algebr. Represent. Theory
  \textbf{3} (2000), no.~3, 213--257.

\bibitem{polishchuk-positselski}
A.~Polishchuk and L.~Positselski, \emph{Quadratic algebras}, University Lecture
  Series, vol.~37, American Mathematical Society, Providence, RI, 2005.

\bibitem{raedschelders-vandenbergh}
T.~Raedschelders and M.~Van den Bergh, \emph{The representation theory of
  noncommutative {$\Oscr(\GL_2)$}}, preprint.

\bibitem{reyes-rogalski-zhang}
M.~Reyes, D.~Rogalski, and J.~Zhang, \emph{Skew {C}alabi-{Y}au algebras and
  homological identities}, Adv. Math. \textbf{264} (2014), 308--354.

\bibitem{riehl-verity}
E.~Riehl and D.~Verity, \emph{The theory and practice of {R}eedy categories},
  Theory Appl. Categ. \textbf{29} (2014), 256--301.

\bibitem{rubtsov-silantyev-talalaev}
V.~Rubtsov, A.~Silantyev, and D.~Talalaev, \emph{Manin matrices, quantum
  elliptic commutative families and characteristic polynomial of elliptic
  {G}audin model}, SIGMA Symmetry Integrability Geom. Methods Appl. \textbf{5}
  (2009), Paper 110, 22.

\bibitem{schauenburg}
P.~Schauenburg, \emph{Hopf bi-{G}alois extensions}, Comm. Algebra \textbf{24}
  (1996), no.~12, 3797--3825.

\bibitem{smith}
S.~P. Smith, \emph{Some finite-dimensional algebras related to elliptic
  curves}, Representation theory of algebras and related topics (Mexico City,
  1994), Amer. Math. Soc., Providence, RI, 1996, pp.~315--348.

\bibitem{vandenbergh}
M.~Van~den Bergh, \emph{Calabi-{Y}au algebras and superpotentials}, Selecta
  Math. (N.S.) \textbf{21} (2015), no.~2, 555--603.

\bibitem{walton-wang}
C.~Walton and X.~Wang, \emph{On quantum groups associated to non-noetherian
  regular algebras of dimension 2}, arXiv:1503.09185 [math.RA].

\end{thebibliography}
%\bibliographystyle{amsplain}
\providecommand{\bysame}{\leavevmode\hbox to3em{\hrulefill}\thinspace}
\providecommand{\MR}{\relax\ifhmode\unskip\space\fi MR }
\providecommand{\MRhref}[2]{%
  \href{http://www.ams.org/mathscinet-getitem?mr=#1}{#2}
}
\providecommand{\href}[2]{#2}

\end{document}